\documentclass[11pt, a4paper]{amsart}
\usepackage{amsmath, graphicx, cite, enumerate, comment}
\usepackage{amssymb, amsthm, amscd}
\usepackage{latexsym}
\usepackage{amssymb}
\usepackage[T1]{fontenc}	
\usepackage[english]{babel}
\usepackage{amsmath,amssymb}
\usepackage{amsfonts}
\usepackage{amscd}
\usepackage{cite}
\usepackage{float}
\usepackage{systeme}
\usepackage{marvosym}
\usepackage{mathrsfs}
\usepackage{amsthm}
\usepackage{hyperref}
\usepackage{verbatim}
\usepackage{caption}
\usepackage{subcaption}
\usepackage{musicography}
\usepackage[utf8]{inputenc}

\usepackage{hyperref}
\hypersetup{
	bookmarks=true,         
	unicode=false,          
	pdftoolbar=true,        
	pdfmenubar=true,        
	pdffitwindow=false,     
	pdfstartview={FitH},    
	pdftitle={My title},    
	pdfauthor={Author},     
	pdfsubject={Subject},   
	pdfcreator={Creator},   
	pdfproducer={Producer}, 
	pdfkeywords={keyword1, key2, key3}, 
	pdfnewwindow=true,      
	colorlinks=true,       
	linkcolor=black,          
	citecolor=blue,        
	filecolor=magenta,      
	urlcolor=cyan           
}

\usepackage{soul}
\setulcolor{gray}
\setul{1pt}{1pt} 

\usepackage{tikz} \usetikzlibrary{arrows,shapes,patterns,cd, arrows,decorations.markings} 
\usepackage{lipsum} 

\theoremstyle{plain}
\newtheorem{theorem}{Theorem}[section]
\newtheorem*{theorem*}{Theorem}
\newtheorem{lemma}[theorem]{Lemma}
\newtheorem*{lemma*}{Lemma}
\newtheorem{proposition}[theorem]{Proposition}
\newtheorem*{proposition*}{Proposition}

\newtheorem*{conjecture*}{Conjecture}
\newtheorem{corollary}[theorem]{Corollary}
\newtheorem*{corollary*}{Corollary}

\theoremstyle{definition}
\newtheorem{definition}[theorem]{Definition}

\newtheorem{remark}[theorem]{Remark}
\newtheorem{claim}{Claim}

\numberwithin{equation}{section}

\newcommand{\R}{\mathbb{R}}
\newcommand{\C}{\mathbb{C}}
\renewcommand{\H}{\mathbb{H}}
\newcommand{\Z}{\mathbb{Z}}
\newcommand{\N}{\mathbb{N}}
\renewcommand{\P}{\mathbb{P}}

\newcommand{\av}[1]{\lvert #1 \rvert}

\newcommand{\Id}{\mathrm{Id}}

\newcommand\restr[2]{{
  \left.\kern-\nulldelimiterspace 
  #1 
  \vphantom{\big|} 
  \right|_{#2} 
  }}
\newcommand{\iprod}{\mathbin{\lrcorner}}
\newcommand{\dS}{{\ooalign{\(S\)\cr\hidewidth\(/\)\hidewidth\cr}}}

\renewcommand{\Im}{\mathrm{Im}}
\renewcommand{\Re}{\mathrm{Re}}

\DeclareMathOperator{\vol}{vol}

\DeclareMathOperator{\Spin}{Spin}
\DeclareMathOperator{\G2}{G_2}
\DeclareMathOperator{\SU}{SU}
\DeclareMathOperator{\SO}{SO}
\DeclareMathOperator{\GL}{GL}
\DeclareMathOperator{\Sp}{Sp}
\DeclareMathOperator{\Hol}{Hol}

\begin{document}
            
            	\title[]
            	{Cayley fibrations in the Bryant-Salamon Spin(7) manifold}
            	\author{Federico Trinca}
            	\address{Mathematical Institute, University of Oxford, Woodstock Road, Oxford, OX2 6GG,
United Kingdom}
            	\email{Federico.Trinca@maths.ox.ac.uk}
                   
            \begin{abstract} 
On each complete asymptotically conical $\Spin(7)$ manifold constructed by Bryant and Salamon, including the asymptotic cone, we consider a natural family of $\SU(2)$ actions preserving the Cayley form. For each element of this family, we study the (possibly singular) invariant Cayley fibration, which we describe explicitly, if possible. These can be reckoned as generalizations of the trivial flat fibration of $\R^8$ and the product of a line with the Harvey–Lawson coassociative fibration of $\R^7$. The fibres will provide new examples of asymptotically conical Cayley submanifolds in the Bryant–Salamon manifolds of topology $\R^4, \R\times S^3$ and $\mathcal{O}_{\C\P^1} (-1)$.
            \end{abstract}
            
            	\maketitle

       \section{Introduction}
In 1926, Cartan showed how to associate a group to any Riemannian manifold through parallel transport \cite{Car26}. He called such a group the holonomy group of the Riemannian manifold, and he used it to classify symmetric spaces. Almost 30 years later, Berger found all the groups that could appear as the holonomy of a simply-connected, nonsymmetric, and irreducible Riemannian manifold \cite{Ber55}. The exceptional holonomy groups $\G2$ and $\Spin(7)$ belonged to this list. The existence of Riemannian manifolds with such holonomy was unknown until Bryant \cite{Bry87} provided incomplete examples and Bryant--Salamon \cite{BS89} provided complete ones. In particular, Bryant and Salamon constructed a $1$-parameter family of torsion-free $\G2$-structures on $\Lambda^2_{-} (T^\ast S^4)$, $\Lambda^2_{-} (T^\ast \C\P^2)$, $\dS (S^3)$, and a $1$-parameter family of torsion-free $\Spin(7)$-structures on $\dS_{-} (S^4)$. The holonomy principle implies that the holonomy group of these manifolds is contained in $\G2$ and $\Spin(7)$, respectively. As Bryant and Salamon proved that their examples have full holonomy, the problem of the classification of Riemannian holonomy groups is settled. 

Manifolds with exceptional holonomy are Ricci-flat and admit natural calibrated submanifolds. These are the associative 3-folds and the coassociative 4-folds in the $\G2$ case, while they are the Cayley 4-folds in the $\Spin(7)$ one. A crucial aspect of the study of manifolds with exceptional holonomy regards fibrations through these natural submanifolds. One of the main reasons for the interest in calibrated fibrations comes from mathematical physics. Indeed, analogously to the SYZ conjecture \cite{SYZ96}, that relates special Lagrangian fibrations in mirror Calabi--Yau manifolds, one would expect similar dualities for coassociative fibrations in the $\G2$ case and Cayley fibrations in the $\Spin(7)$ one. We refer the reader to \cite{GYZ03, Ach98} for further details. Another reason lies in the attempt to understand and construct new compact manifolds with exceptional holonomy through these fibrations \cite{Don17}. 

Some work has been carried out in the $\G2$ case (see f.i. \cite{ABS20, Bar10, Don17, KLo21, Li19}), while little is known in the $\Spin(7)$ setting. In particular, Karigiannis and Lotay \cite{KLo21} constructed an explicit coassociative fibration on each $\G2$ Bryant--Salamon manifold and the relative asymptotic cone. To do so, they chose a $3$-dimensional Lie group acting through isometries preserving the $\G2$-structure, and they imposed the fibres to be invariant under this group action. In this way, the coassociative condition is reduced to a system of tractable ODEs defining the fibration. Previously, this idea was used to study cohomogeneity one calibrated submanifolds related to exceptional holonomy in the flat case by Lotay \cite{Lot07} and in $\Lambda^2_{-} (T^\ast S^4)$ by Kawai \cite{Kaw18}. Analogously, we consider Cayley fibrations on each $\Spin(7)$ Bryant--Salamon manifold and the relative asymptotic cone, which are invariant under a natural family of structure-preserving $\SU(2)$ actions.

The first key observation, due to Bryant and Salamon \cite{BS89}, is that $\Sp(2)\times\Sp(1)$ is contained in the subgroup of the isometry group that preserves the $\Spin(7)$-structure. Indeed, one can lift an action of $\SO(5)$ on $S^4$ to an action of $\Spin(5)\cong\Sp(2)$ on the spinor bundle of $S^4$. The $\Sp(1)$ factor of $\Sp(2)\times\Sp(1)$ comes from a twisting of the fibre. Clearly, this group admits plenty of $3$-dimensional subgroups. The family we consider consists of the subgroups that respect the direct product, i.e. that do not sit diagonally in $\Sp(2)\times\Sp(1)$. Through Lie group theory, it is easy to find these subgroups. Indeed, they either are the whole $\Sp(1)$, appearing in the second factor or the lift of one of the following subgroups of $\SO(5)$, which are going to be contained in the first factor:
\begin{align*}
&\SO(3)\times \Id_2, \hspace{30pt} \Sp(1)\times\Id_1,\\
&\SO(3)\hspace{5pt} \textup{acting irreducibly on } \R^5,
\end{align*}
where $\Sp(1)\times\Id_1$ denotes both the subgroup acting on $\H\times\R$ by left multiplication and by right multiplication of the quaternionic conjugate. Observe that their lifts to $\Sp(2)$ are all diffeomorphic to $\SU(2)\cong\Sp(1)$. Moreover, the $\Sp(1)$ contained in the second factor will only act on the fibres of $\dS_{-} (S^4)$, leaving the base fixed.

\subsection*{Summary of results and organization of the paper}
In section \ref{Section prelims}, we briefly review some basic results on $\Spin(7)$ and Riemannian geometry. In particular, once fixed the convention for the $\Spin(7)$-structure, we recall the definition of Cayley submanifolds, together with Karigiannis--Min-Oo's characterization \cite[Proposition 2.5]{KM05}, and Cayley fibrations. Similarly to \cite[Definition 1.2]{KLo21}, our notion of Cayley fibrations allows the fibres to be singular and to self-intersect. Finally, we provide the definitions of asymptotically conical and conically singular manifolds.

Section \ref{Section BS} contains a detailed description of the 1-parameter family of $\Spin(7)$ manifolds constructed by Bryant--Salamon. Here, we also discuss the automorphism group. In particular, we briefly explain why the system of ODEs characterizing the fibration induced by the irreducible action of $\SO(3)$ on $S^4$ is going to be too complicated to be solved. 

Starting from section \ref{Section fibre action}, we deal with Cayley fibrations. Here, we study the fibration invariant under the $\SU(2)$ acting only on the fibres of $\dS_{-} (S^4)$. In this case, the fibration is trivial, i.e. coincide with the usual projection map from $\dS_{-} (S^4)$ to $S^4$. We compute the multi-moment map in the sense of \cite{MS12,MS13}, which is a polynomial depending on the square of the distance function. Blowing-up at any point of the zero section, the fibration becomes the trivial flat fibration of $\R^8$. 

In section \ref{Section SO(3)xId_2}, we consider the action on $\dS_{-} (S^4)$ induced by $\SO(3)\times \Id_2\subset \SO(5)$ acting on $S^4$. Under a suitable choice of metric-diagonalizing coframe on an open, dense set $\mathcal{U}$, the system of ODEs characterizing the Cayley condition is completely integrable, and hence we obtain a locally trivial fibration on $\mathcal{U}$ whose fibres are Cayley submanifolds. Extending by continuity the fibration to the whole $\dS_{-} (S^4)$, we prove that the parameter  is $S^4$ and the fibres are topological $\R^4$s, $\mathcal{O}_{\C\P^1} (-1)$s or $\R\times S^3$s. Through a asymptotic analysis, it is easy to see that the $\R^4$s separating the Cayleys of different topology are the only singular ones. The singularity is asymptotic to the Lawson--Osserman cone \cite{LO77}. Each Cayley intersects at least another one in the zero section of $\dS_{-} (S^4)$, and, at infinity, they are asymptotic to a non-flat cone with link $S^3$ endowed with either the round metric or a squashed metric. While in the $\G2$ case \cite[Subsection 5.7, Subsection 6.7]{KLo21}, the multi-moment map they explicitly compute has a clear geometrical interpretation, it does not in our case. Finally, keeping track of the Cayley fibration, blowing-up at the north pole, we obtain the fibration on $\R^8$, which is given by the product of the $\SU(2)$-invariant coassociative fibration constructed by Harvey and Lawson \cite[Section IV.3]{HL82} with a line. 

We deal with the Cayley fibration invariant under the $\SU(2)$ action induced from $\Sp(1)\times\Id_1$ in Section \ref{Section Sp(1)xId_1}. The left quaternionic multiplication gives the same fibration as the conjugate right quaternionic multiplication up to orientation. Contrary to the previous case, we can not completely integrate the system of ODEs we obtain on an open, dense set $\mathcal{U}$. However, we deduce all the information we are interested in via a dynamical system argument. In particular, we show that the fibres are parametrized by a $4$-dimensional sphere and that they are smooth submanifolds of topology $S^3\times\R$, $\R^4$ or $S^4$. The unique point of intersection is the south (north) pole of the zero section, where all fibres of topology $\R^4$ and the sole Cayley of topology $S^4$ (i.e. the zero section) intersect. It is easy to show that all Cayleys are asymptotic to a non-flat cone with round link $S^3$. We also compute the multi-moment map, and show that the fibration converges to the trivial flat fibration of $\R^8$ when we blow-up at the north pole. 

The last group action that would be natural to study is the lift of $\SO(3)$ acting irreducibily on $\R^5$. However, in this case the ODEs become extremely complicated and can not be solved explicitely. Moreover, the analogous action on the flat $\Spin(7)$ space and on the Bryant-Salamon $\G2$ manifold $\Lambda^2_{-} (T^\ast S^4)$ was studied by Lotay \cite[Subsection 5.3.3]{Lot05} and Kawai \cite{Kaw18}, respectively. In both cases, the defining ODEs for Cayley submanifolds and coassociative submanifolds were too complicated.

\subsection*{Acknowledgements}
The author wishes to thank his supervisor Jason D. Lotay for suggesting this project and for his enormous help and guidance. He also wishes to thank the referee for carefully reading an earlier version of this paper and for greatly improving its exposition. This work was supported by the Oxford-Thatcher Graduate Scholarship.

\section{Preliminaries}\label{Section prelims}
In this section, we recall some basic results concerning $\Spin(7)$ manifolds, Cayley submanifolds and Riemannian conifolds.
\subsection{Spin(7) manifolds}\label{Spin(7) structures}
We use the same convention of \cite{BS89} and \cite{HL82} to define $\Spin(7)$-structures and $\Spin(7)$ manifolds. 

The local model is $\R^8\cong\R^4\oplus\R^4$ with coordinates $(x_0,...,x_3,a_0,...,a_3)$, and Cayley form: 
\[
\Phi_{\R^8}=d x_0\wedge d x_1\wedge d x_2\wedge d x_3+ d a_0\wedge d a_1\wedge d a_2\wedge d a_3+\sum_{i=1}^3 	\omega_i\wedge\eta_i,
\]
where $\omega_i=d x_0\wedge dx_i-dx_j\wedge dx_k$, $\eta_i=d a_0\wedge da_i-da_j\wedge da_k$ and $(i,j,k)$ is a cyclic permutation of $(1,2,3)$. Note that $\{\omega_i\}_{i=1}^3$ and $\{\eta_i\}_{i=1}^3$ are the standard basis of the anti-self-dual 2-forms on the two copies of $\R^4$. It is well-known that $\Spin(7)$ is isomorphic to the stabilizer of $\Phi_{\R^8}$ in $\mathrm{GL}(8,\R)$.
\begin{remark}
This choice of convention for $\Phi_{\R^8}$ is compatible with the fact that we will be working on $\dS_{\!-}(S^4)$. Indeed, we can identify our local model with $\dS_{\!-}(\R^4)$. Further details regarding the sign conventions and orientations for $\Spin(7)$-structures can be found in \cite{Kar10}.
\end{remark}

\begin{definition}
Let $M$ be a manifold and let $\Phi$ be a 4-form on $M$. We say that $\Phi$ is admissible if, for every $x\in M$, there exists an oriented isomorphism $i_x:\R^8\to T_x M$ such that $i_x^{\ast} \Phi=\Phi_{\R^8}$. We also refer to $\Phi$ as a $\Spin(7)$-structure on $M$.  
\end{definition}
The $\Spin(7)$-structure on $M$ also induces a Riemannian metric, $g_\Phi$, and an orientation, $\vol_\Phi$, on $M$. With respect to these structures $\Phi$ is self-dual. We refer the reader to \cite{SW17} for further details.
\begin{definition}
Let $M$ be a manifold and let $\Phi$ be a $\Spin(7)$-structure on $M$. We say that $(M,\Phi)$ is a $\Spin(7)$ manifold if the $\Spin(7)$-structure is torsion-free, i.e., $d\Phi=0$. In this case, $\Hol(g_\Phi)\subseteq \Spin(7)$.
\end{definition}

\subsection{Cayley submanifolds and Cayley fibrations} \label{subsection cayley submanifolds and cayley fibration}
Given $(M,\Phi)$, $\Spin(7)$ manifold, it is clear that $\Phi$ has comass one, and hence, it is a calibration.
\begin{definition}
We say that a 4-dimensional oriented submanifold is Cayley if it is calibrated by $\Phi$, i.e., if $\restr{\Phi}{N^4}=\vol_{N^4}$. Fixed a point $p\in M$, a 4-dimensional oriented vector subspace $H$ of $T_p M$ is said to be a Cayley 4-plane if $\restr{\Phi}{p}$ calibrates $H$. 
\end{definition}

\begin{remark}
Observe that $N$ is a Cayley submanifold if and only if $T_p N$ is a Cayley 4-plane for all $p\in N$.
\end{remark}

We now give Karigiannis and Min-Oo characterization of the Cayley condition.

\begin{proposition}[Karigiannis--Min-Oo {\cite[Proposition 2.5]{KM05}}] \label{Cayley condition KM}
The subspace spanned by tangent vectors $u,v,w,y$ is a Cayley 4-plane, up to orientation, if and only if the following form vanishes: 
\[
\eta=\pi_7\left(u^{\musFlat}\wedge B(v,w,y)+v^{\musFlat}\wedge B(w,u,y)+w^{\musFlat}\wedge B(u,v,y)+y^{\musFlat}\wedge B(v,u,w)\right), 
\]
where $$B(u,v,w):=w\iprod v\iprod u\iprod \Phi$$ and $$\pi_7(u^{\musFlat}\wedge v^{\musFlat}):=\frac{1}{4} \left(u^{\musFlat}\wedge v^{\musFlat}+u\iprod v\iprod \Phi\right).$$
\end{proposition}
\begin{remark}\label{Remark explanation pi_7}
The reduction of the structure group of $M$ to $\Spin(7)$ induces an orthogonal decomposition of the space of differential $k$-forms for every $k$, which corresponds to an irreducible representation of $\Spin(7)$. In particular, if $k=2$, the irreducible representations of $\Spin(7)$ are of dimension $7$ and $21$. At each point $x\in M$, these representations induce the decomposition of $\Lambda^2 (T^{\ast}_x M)$ into two subspaces, which we denote by $\Lambda^2_7$ and $\Lambda^2_{21}$, respectively. The map $\pi_7$ defined in Proposition \ref{Cayley condition KM} is precisely the projection map from the space of two-forms to $\Lambda^2_7$. Further details can be found in \cite{SW17}.
\end{remark}
Following \cite{KLo21}, we extend the definition of Cayley fibration so that it may admit intersecting fibres and singular fibres. 
\begin{definition}\label{Cayley fibration definition}
Let $(M,\Phi)$ be a $\Spin(7)$ manifold. $M$ admits a Cayley fibration if there exists a family of Cayley submanifolds $N_b$ (possibly singular) parametrized by a $4$-dimensional space $\mathcal{B}$ satisfying the following properties: 
\begin{itemize}
\item $M$ is covered by the family $\{N_b\}_{b\in\mathcal{B}}$;
\item there exists an open dense set $\mathcal{B}^{\circ}\subset\mathcal{B}$ such that $N_b$ is smooth for all $b\in \mathcal{B}^{\circ}$;
\item there exists an open dense set $M'\subset M$ and a smooth fibration $\pi: M'\to \mathcal{B}$ with fibre $N_b$ for all $b\in\mathcal{B}$.
\end{itemize}
\end{definition}
\begin{remark}
The last point allows the Cayley submanifolds in the family $\mathcal{B}$ to intersect. Indeed, this may happen in $M\setminus M'$. Moreover, we may lose information (e.g. completeness and topology) when we restrict the Cayley fibres to $M'$.
\end{remark}

We conclude this subsection explaining how we determine the topology of $\R^2$ bundles over $S^2\cong\C\P^1$ arising as the smooth fibres of a Cayley fibration. This is the same discussion used in \cite{KLo21}. Let $N$ be the total space of an $\R^2$-bundle over $\C\P^1$ which is also a Cayley submanifold of a $\Spin(7)$ manifold $(M,\Phi)$. Since $N$ is orientable and it is the total space of a bundle over an oriented base, it is an orientable bundle. We deduce that $N$ is homeomorphic to a holomorphic line bundle over $\C\P^1$. These objects are classified by an integer $k\in\Z$ and are denoted by $\mathcal{O}_{\C\P^1}(k)$. Moreover, for $k>0$ we have the following topological characterization of $\mathcal{O}_{\C\P^1}(-k)$:
\begin{align*}
\mathcal{O}_{\C\P^1}(-k)\setminus \C\P^1\cong \C^2/\Z_k\cong \R^+\times(S^3/\Z_k).
\end{align*} 
In the situation we will consider, the submanifolds we construct have the form $N\setminus S^2=\R^+\times S^3$. Hence, the only possibility is to obtain topological $\mathcal{O}_{\C\P^1} (-1)$s.

\subsection{Riemannian conifolds}We now recall the definitions of asymptotically conical and conically singular  Riemannian manifolds.
\begin{definition}
A Riemannian cone is a Riemannian manifold $(M_0,g_0)$ with $M_0=\R^+ \times \Sigma$ and $g_0=dr^2+r^2 g_\Sigma$, where $r$ is the coordinate on $\R^+$ and $g_\Sigma$ is a Riemannian metric on the link of the cone, $\Sigma$. 
\end{definition}
\begin{definition}
We say that a Riemannian manifold $(M,g)$ is asymptotically conical (AC) with rate $\lambda<0$ if there exists a Riemannian cone $(M_0,g_0)$ and a diffeomorphism $\Psi: (R,\infty)\times \Sigma \to M\setminus K$ satisfying: 
\[
\av{\nabla^j (\Psi^\ast g-g_0)}=O(r^{\lambda-j}) \hspace{15pt}  r\rightarrow\infty \hspace{5pt} \forall j\in \N,
\]
where $K$ is a compact set of $M$ and $R>0$. $(M_0,g_0)$ is the asymptotic cone of $(M,g)$ at infinity.
\end{definition}

\begin{definition}
We say that a Riemannian manifold $(M,g)$ is conically singular with rate $\mu>0$ if there exists a Riemannian cone $(M_0,g_0)$ and a diffeomorphism $\Psi: (0,\epsilon)\times \Sigma \to M\setminus K$ satisfying: 
\[
\av{\nabla^j (\Psi^\ast g-g_0)}=O(r^{\mu-j}) \hspace{15pt}  r\rightarrow 0 \hspace{5pt} \forall j\in \N,
\]
where $K$ is a closed subset of $M$ and $\epsilon>0$. $(M_0,g_0)$ is the asymptotic cone of $(M,g)$ at the singularities.
\end{definition}

\begin{remark}
As $\Sigma$ does not need to be connected, AC manifolds may admit more than one end and asymptotically singular manifolds may admit more than one singular point. 
\end{remark}

\section{Bryant--Salamon \texorpdfstring{$\Spin(7)$}{Lg} manifolds}\label{Section BS}
In this section we will describe the central objects of this work, i.e., the $\Spin(7)$ manifolds constructed by Bryant and Salamon in \cite{BS89}. There, they provided a $1$-parameter family of torsion-free $\Spin(7)$-structures on $M:=\dS_{\!-}(S^4)$, the negative spinor bundle on $S^4$. The $4$-dimensional sphere is endowed with the metric of constant sectional curvature $k$, which is the unique spin self-dual Einstein $4$-manifold with positive scalar curvature \cite{Hit81}. Without loss of generality, we rescale the sphere so that $k=1$.

\begin{remark}
The Bryant--Salamon construction on $S^4$ also works on spin $4$-manifolds with self-dual Einstein metric, but negative scalar curvature, and on spin orbifolds with self-dual Einstein metric. However, in these cases, the metric is not complete or smooth. 
\end{remark}

\subsection{The negative spinor bundle of \texorpdfstring{$S^4$}{Lg}}
Let $S^4$ be the 4-sphere endowed with the Riemannian metric of constant sectional curvature $1$. As $S^4$ is clearly spin, given $P_{\SO(4)}$ frame bundle of $S^4$ we can find the spin structure $P_{\Spin(4)}$ together with the spin representation: 
\[
\mu:=(\mu_{+},\mu_{-}):\Sp(1)\times\Sp(1)\cong\Spin(4)\to\GL(\H)\times\GL(\H),
\]
where $\mu_{\pm}(p_{\pm})(v):=v\overline{p}_{\pm}$. Let $\tilde{\pi}:P_{\Spin(4)}\to P_{\SO(4)}$ be the double cover in the definition of spin structure, and let $\tilde{\pi}^n_0:\Spin(n)\to\SO(n)$ be the double (universal) covering map for all $n\geq3$. The negative spinor bundle over $S^4$ is defined as the associated bundle:
$$\dS_{\!-}(S^4):=P_{\Spin(4)}\times_{\mu_{-}} \H.$$
The positive spinor bundle is defined analogously, taking $\mu_{+}$ instead. 

Given an oriented local orthonormal frame for $S^4$, $\{e_0,e_1,e_2,e_3\}$, the real volume element $e_0\cdot e_1\cdot e_2\cdot e_3$ acts as the identity on the negative spinors and as minus the identity on the positive ones. Now, let $\{b_0,b_1,b_2,b_3\}$ be the dual coframe of $\{e_0,e_1,e_2,e_3\}$, let $\tilde{\omega}$ the connection 1-forms relative to the Levi-Civita connection of $S^4$ with respect to the frame $\{e_0,e_1,e_2,e_3\}$ and let $\{\sigma_1, \sigma_i, \sigma_j, \sigma_k\}$ a local orthonormal frame for the negative spinor bundle corresponding to the standard basis of $\{1,i,j,k\}$ in this trivialization. Hence, we can define the linear coordinates $(a_0, a_1, a_2, a_3)$ which parametrize a point in the fibre as $a_0 \sigma_1+a_1 \sigma_i+a_2 \sigma_j+a_3 \sigma_k$.

By the properties of the spin connection and the fact we are working on the negative spinor bundle, we can write: 
\begin{align*}
\nabla \sigma_{\alpha}&=\left(\rho_1 \mu_{-} (e_2\cdot e_3)+\rho_2 \mu_{-} (e_3\cdot e_1)+\rho_3 \mu_{-} (e_1\cdot e_2) \right)\sigma_{\alpha} \\
&=\left(\rho_1 \mu_{-} (i)+\rho_2 \mu_{-} (j)+\rho_3 \mu_{-} (k) \right) \sigma_{\alpha},
\end{align*}
where $2\rho_1=\tilde{\omega}^3_2-\tilde{\omega}^1_0$, $2\rho_2=-\tilde{\omega}^2_0-\tilde{\omega}^3_1$ and $2\rho_3=\tilde{\omega}^2_1-\tilde{\omega}^3_0$. It is well-known that these are the connection forms on the bundle of anti-self-dual 2-forms, with respect to the connection induced by the Levi-Civita connection on $S^4$ and the frame given by $\Omega_i:=b_0\wedge b_i-b_j\wedge b_k$. As usual, $(i,j,k)$ is a cyclic permutation of $(1,2,3)$. The $\rho_i$s are characterized by: 
\begin{align}
d 
\begin{pmatrix}
\Omega_1 \\ \Omega_2 \\ \Omega_3
\end{pmatrix}
=-
\begin{pmatrix}
0 & -2\rho_3 & 2\rho_2 \\ 2\rho_3 &0&-2\rho_1 \\ -2\rho_2&2\rho_1&0
\end{pmatrix}
\wedge\begin{pmatrix}
\Omega_1 \\ \Omega_2 \\ \Omega_3
\end{pmatrix}, \label{comprho}
\end{align}
and the vertical one forms are: 
\begin{equation}\label{compVert1Forms}
\begin{aligned}
\xi_0&=da_0+\rho_1 a_1+ \rho_2 a_2+\rho_3 a_3, \hspace{20pt} &&\xi_1=da_1-\rho_1 a_0- \rho_3 a_2+\rho_2 a_3,\\
\xi_2&=da_2-\rho_2 a_0+ \rho_3 a_1-\rho_1 a_3,  &&\xi_3=da_3-\rho_3 a_0- \rho_2 a_1+\rho_1 a_2.
\end{aligned}
\end{equation}
\begin{remark}
If we denote by $\pi_{S^4}$ the vector bundle projection map from $\dS_{\!-}(S^4)$ to $S^4$, we can obtain horizontal forms on $\dS_{\!-}(S^4)$ via pull-back. For example, $\{\pi_{S^4}^{\ast}(b_i)\}_{i=1}^4$ and the linear combinations of their wedge product are horizontal forms on $\dS_{\!-}(S^4)$. In order to keep our notation light, we will omit the pullback from now on. 
\end{remark}
As $S^4$ is self-dual and with scalar curvature equal to $12k$, we have: 
\begin{align*}
d 
\begin{pmatrix}
\rho_1 \\ \rho_2 \\ \rho_3
\end{pmatrix}
=-2
\begin{pmatrix}
\rho_2\wedge\rho_3 \\ \rho_3\wedge\rho_1 \\ \rho_1\wedge\rho_2
\end{pmatrix}
+\frac{1}{2}\begin{pmatrix}
\Omega_1 \\ \Omega_2 \\ \Omega_3
\end{pmatrix},
\end{align*}
which is equivalent to \cite[p. 842]{BS89} and \cite[3.24]{KLo21}. We can use it to compute:
\begin{equation}\label{dxi_i}
\begin{aligned}
d\xi_0&=\xi_1\wedge\rho_1+\xi_2\wedge\rho_2+\xi_3\wedge\rho_3+1/2\left(a_1\Omega_1+a_2\Omega_2+a_3\Omega_3\right), \\
d\xi_1&=-\xi_0\wedge\rho_1-\xi_2\wedge\rho_3+\xi_3\wedge\rho_2+1/2\left(-a_0\Omega_1-a_2\Omega_3+a_3\Omega_2\right), \\
d\xi_2&=-\xi_0\wedge\rho_2+\xi_1\wedge\rho_3-\xi_3\wedge\rho_1+1/2\left(-a_0\Omega_2+a_1\Omega_3-a_3\Omega_1\right), \\
d\xi_3&=-\xi_0\wedge\rho_3-\xi_1\wedge\rho_2+\xi_2\wedge\rho_1+1/2\left(-a_0\Omega_3-a_1\Omega_2+a_2\Omega_1\right),
\end{aligned}
\end{equation}
that is going to be useful below. 
\begin{remark}
 A detailed account of spin geometry can be found in \cite{LM89}. Observe that, there, the definition of positive and negative spinors is interchanged. We opted to stay consistent with \cite{BS89}.
Indeed, the vertical 1-forms we obtain coincide with the ones obtained by Bryant and Salamon, up to renaming the $\rho_i$s. The same holds for the relative exterior derivatives.
\end{remark}
\subsection{The \texorpdfstring{$\Spin(7)$}{Lg}-structures}
If $r^2:=a_0^2+a_1^2+a_2^2+a_3^2$ is the square of the distance function from the zero section and  $c$ is a positive constant, then, the $\Spin(7)$-structures defined by Bryant and Salamon are: 
\begin{equation} \label{Phi_c}
\begin{aligned}
\Phi_c:=&16(c+r^2)^{-4/5}\xi_0\wedge\xi_1 \wedge\xi_2\wedge\xi_3+25 (c+r^2)^{6/5}b_0\wedge b_1\wedge b_2\wedge b_3\\
&+20(c+ r^2)^{1/5} (A_1\wedge\Omega_1+A_2\wedge\Omega_2+A_3\wedge\Omega_3),
\end{aligned}
\end{equation}
where $A_i:=\xi_0\wedge \xi_i -\xi_j\wedge \xi_k$. As usual, $(i,j,k)$ is a cyclic permutation of $(1,2,3)$.

The metric induced by $\Phi_{c}$ is
\begin{align}\label{g_c}
g_c:=4(c+ r^2)^{-2/5} (\xi_0^2+\xi_1^2+\xi_2^2+\xi_3^2)+5(c+r^2)^{3/5} (b_0^2+b_1^2+b_2^2+b_3^2),
\end{align}
while the induced volume element is
\begin{align}\label{Vol_c}
\vol_c:=\left(20\right)^2 (c+r^2)^{2/5}(\xi_0\wedge\xi_1\wedge\xi_2\wedge\xi_3\wedge b_0\wedge b_1\wedge b_2 \wedge b_3).
\end{align}

Setting $c=0$ and $M_0:=\dS_{\!-}(S^4)\setminus S^4\cong\R^{+}\times S^7$, we obtain a $\Spin(7)$ cone $(M_0,\Phi_0)$, i.e. $M_0$ with the metric induced by the $\Spin(7)$-structure $\Phi_0$ is a Riemannian cone.

\begin{theorem}[Bryant--Salamon {\cite[p. 847]{BS89}}]
The $\Spin(7)$-structure $\Phi_c$ is torsion-free for all $c\geq0$. Moreover, these manifolds have full holonomy $\Spin(7)$.
\end{theorem}
 
 It is well-known that the Bryant--Salamon $\Spin(7)$ manifolds we have just described are asymptotically conical (see for instance \cite[p.184]{Sal89}), hence, we state here the main results concerning their asymptotic geometry.
\begin{theorem}
For every $c\geq0$, $(M,\Phi_c)$ is an asymptotically conical Riemannian manifold with rate $\lambda=-10/3$ and asymptotic cone $(M_0,\Phi_0)$. 
\end{theorem}

\subsection{Automorphism Group}
A natural subset of the diffeomorphism group of a $\Spin(7)$-manifold is the automorphism group, i.e. the subgroup that preserves the $\Spin(7)$-structure. Clearly, the automorphism group is contained in the isometry group with respect to the induced metric.

In the setting we are considering, Bryant and Salamon noticed that the diffeomorphisms given by the $\Sp(2)\times\Sp(1)$-action described as follows are actually in the automorphism group \cite[Theorem 2]{BS89}. Consider $\SO(5)$ acting on $S^4$ in the standard way. This induces an action on the frame bundle of $S^4$ via the differential, which easily lifts to a $\Spin(5)\cong\Sp(2)$ action on $P_{\Spin(4)}$. If we combine it with the standard quaternionic left-multiplication by unit vectors on $\H$, we have defined an $\Sp(2)\times\Sp(1)$ action on $P_{\Spin(4)}\times\H$. As it commutes with $\mu_{-}$, it passes to the quotient $\dS_{\!-}(S^4)$.

By Lie group theory \cite[Appendix B]{Kaw18}, we know that the $3$-dimensional connected closed subgroups of $\Sp(2)$ are the lift of one of the following subgroups of $\SO(5)$:
\begin{align*}
&\SO(3)\times \Id_2, \hspace{30pt} \Sp(1)\times\Id_1,\\
&\SO(3)\hspace{5pt} \textup{acting irreducibly on } \R^5,
\end{align*}
where $\Sp(1)\times\Id_1$ denotes both the subgroup acting on $\H\times\R$ by left multiplication and by right multiplication of the quaternionic conjugate. Observe that they are all diffeomorphic to $\SU(2)$. In particular, the family of $3$-dimensional subgroups that do not sit diagonally in $\Sp(2)\times\Sp(1)$ consists of 
$$G\times 1_{\Sp(1)}\subset\Sp(2)\times\Sp(1)$$
 and 
 $$1_{\Sp(2)}\times\Sp(1)\subset \Sp(2)\times \Sp(1),$$
  where $G$ is one of the lifts above. These are going to be the subgroups of the automorphism group that we will take into consideration.
  
\section{The Cayley fibration invariant under the \texorpdfstring{$\Sp(1)$}{Lg} action on the fibre} \label{Section fibre action}

Let $M:=\dS_{\!-}(S^4)$ and $M_0:=\R^{+}\times S^7$ endowed with the torsion-free $\Spin(7)$-structures $\Phi_c$ constructed by Bryant and Salamon and described in Section \ref{Section BS}.

Observe that $(M,\Phi_c)$ and $(M_0,\Phi_0)$ admit a trivial Cayley Fibration. Indeed, it is straightforward to see that the natural projection to $S^4$ realizes both spaces as honest Cayley fibrations with smooth fibres diffeomorphic to $\R^4$ and $\R^4\setminus\{0\}$, respectively. In both cases, the parametrizing family is clearly $S^4$

The fibres are asymptotically conical to the cone of link $S^3$ and metric: 
\begin{align*}
ds^2+\frac{9}{25} g_{S^3},
\end{align*}
where $s= r^{3/5} 10/3$ and $g_{S^3}$ is the standard unit round metric. 

Since $\Id_{\Sp(2)}\times\Sp(1)$ acts trivially on the basis, and as $\Sp(1)$ on the fibres of $\dS_{\!-}(S^4)$ identified with $\H$, it is clear that the trivial fibration is invariant under $\Id_{\Sp(2)}\times\Sp(1)$. 

\begin{remark}
We compute the associated multi-moment map, $\nu_c$, in the sense of Madsen and Swann \cite{MS12,MS13}. This is:
\begin{align*}
\nu_c:=\frac{20}{3}(r^2-5c)(c+r^2)^{1/5}+\frac{100}{3}c^{6/5},
\end{align*}
where we subtracted $c^{6/5}100/3$ so that the range of the multi-moment map is $[0,\infty)$. Observe that the level sets of $\nu_c$ coincide with the level sets of the distance function from the zero section.
\end{remark}
\begin{remark}
As in \cite[Section 4.4]{KLo21}, this fibration becomes the trivial Cayley fibration of $\R^8=\R^4\oplus\R^4$ when we blow-up at any point of the zero section.
\end{remark}

\section{The Cayley fibration invariant under the lift of the \texorpdfstring{$\SO(3)\times\Id_2$}{Lg} action on \texorpdfstring{$S^4$}{Lg}} \label{Section SO(3)xId_2}

Let $M:=\dS_{\!-}(S^4)$ and $M_0:=\R^{+}\times S^7$ be endowed with the torsion-free $\Spin(7)$-structures $\Phi_c$ constructed by Bryant and Salamon that we described in Section \ref{Section BS}. On each $\Spin(7)$ manifold, we construct the Cayley Fibration which is invariant under the lift to $M$ (or $M_0$) of the standard $\SO(3)\times\Id_2$ action on $S^4\subset \R^3\oplus\R^2$. 

\subsection{The choice of coframe on \texorpdfstring{$S^4$}{Lg}}\label{Subsection Frame S^4 SO(3)xId_2}
As in \cite{KLo21}, we choose an adapted orthonormal coframe on $S^4$ which is compatible with the symmetries we will impose. Since the action coincides, when restricted to $S^4$, with the one used by Karigiannis and Lotay on $\Lambda^2_{-} (T^{\ast}S^4)$ \cite[Section 5]{KLo21}, it is natural to employ the same coframe, which we now recall.

We split $\R^5$ into the direct sum of a $3$-dimensional vector subspace $P\cong\R^3$ and its orthogonal complement $P^{\perp}\cong\R^2$. As $S^4$ is the unit sphere in $\R^5$, we can write, with respect to this splitting:
\begin{align*}
S^4=\left\{(\textbf{x},\textbf{y})\in P\oplus P^{\perp} : \av{\textbf{x}}^2+\av{\textbf{y}}^2=1\right\}.
\end{align*}
Now, for all $(\textbf{x},\textbf{y})\in S^4$ there exists a unique $\alpha\in[0,\pi/2]$, some $\textbf{u}\in S^2\subset P$ and some $\textbf{v}\in S^1\subset P^{\perp}$ such that: 
\begin{align*}
\textbf{x}=\cos \alpha \textbf{u}, \hspace{20pt} \textbf{y}=\sin \alpha \textbf{v}.
\end{align*}
Observe that $\textbf{u}$ and $\textbf{v}$ are uniquely determined when $\alpha\in(0,\pi/2)$, while, when $\alpha=0,\pi/2$, $\textbf{v}$ can be any unit vector in $P^{\perp}$ ($\textbf{y}=0$) and $\textbf{u}$ can be any unit vector in $P$ ($\textbf{x}=0$), respectively. Hence, we are writing $S^4$ as the disjoint union of an $S^2$, corresponding to $\alpha=0$, of an $S^1$, corresponding to $\alpha=\pi/2$, and of $S^2\times S^1\times (0,\pi/2)$. 

If we put spherical coordinates on $S^2$ and polar coordinates on $S^1$, then, we can write
\begin{align*}
\textbf{u}=(\cos\theta,\sin\theta\cos\phi,\sin\theta\sin\phi),
\end{align*}
and
\begin{align*}
\textbf{v}=(\cos\beta,\sin\beta),
\end{align*}
where $\theta\in[0,\pi]$, $\phi\in[0,2\pi)$ and $\beta\in[0,2\pi)$. As usual, $\phi$ is not unique when $\theta=0,\pi$.

It follows that, if we take out the points where $\theta=0,\pi$ from $S^2\times S^1\times (0,\pi/2)$, we have constructed a coordinate patch $U$ parametrized by $(\alpha,\beta,\theta,\phi)$ on $S^4$. Explicitly, $U$ is $S^4$ minus two totally geodesic $S^2$:
\begin{align*}
S^2_{y_1,y_2=0} =\left\{ (\textbf{x},\textbf{0})\in P\oplus P^{\perp} : \av{x}^2=1\right\},
\end{align*}
corresponding to $\alpha=0$, and 
\begin{align*}
S^2_{x_2,x_3=0}=\left\{(\cos\alpha,0,0,\sin\alpha\cos\beta,\sin\alpha\sin\beta)\in P\oplus P^{\perp} : \alpha\in(0,\pi)\right\},
\end{align*}
corresponding to $\theta=0$ and $\theta=\pi$. Observe, that the $S^1$ corresponding to $\alpha=\pi/2$ is a totally geodesic equator in $S^2_{x_2,x_3=0}$.

A straightforward computation shows that the coordinate frame $\{\partial_{\alpha},\partial_{\beta},\partial_{\theta},\partial_{\phi}\}$ is orthogonal and can be easily normalized obtaining: 
\begin{align*}
f_0:=\partial_{\alpha}, \hspace{15pt}f_1:=\frac{\partial_{\beta}}{\sin\alpha},\hspace{15pt}f_2:=\frac{\partial_{\theta}}{\cos\alpha},\hspace{15pt}f_3:=\frac{\partial_{\phi}}{\cos\alpha\sin\theta}.
\end{align*}
The dual orthonormal coframe is given by: 
\begin{align}\label{frame S^4 SO(3)xId_2}
b_0:=d\alpha,\hspace{15pt} b_1:=\sin\alpha d\beta,\hspace{15pt} b_2:=\cos\alpha d\theta, \hspace{15pt} b_3:=\cos\alpha\sin\theta d\phi.
\end{align}
Observe that $\{b_0,b_1,b_2,b_3\}$ is positively oriented with respect to the outward pointing normal of $S^4$, hence, the volume form is: 
\begin{align*}
\vol_{S^4}=\sin\alpha\cos^2\alpha\sin\theta d\alpha\wedge d\beta\wedge d\theta\wedge d\phi.
\end{align*}
\subsection{The horizontal and the vertical space} \label{Subsection horizontal and vertical space SO(3)xId_2}
As in \cite[Subsection 5.2]{KLo21}, we use (\ref{comprho}) to compute the $\rho_i$'s in the coordinate frame we have just defined. Indeed, (\ref{frame S^4 SO(3)xId_2}) implies that:
\begin{equation}\label{Omega_i on S^4 SO(3)xId_2}
\begin{aligned}
\Omega_1&=\sin\alpha d\alpha\wedge d\beta-\cos^2\alpha \sin\theta d\theta\wedge d\phi,\\
\Omega_2&=\cos\alpha d\alpha\wedge d\theta-\sin\alpha\cos\alpha \sin\theta d\phi\wedge d\beta,\\
\Omega_3&=\cos\alpha\sin\theta d\alpha\wedge d\phi-\sin\alpha \cos\alpha d\beta\wedge d\theta;
\end{aligned}
\end{equation}
hence, we deduce that: 
\begin{equation*}
\begin{aligned}
d\Omega_1&=2\sin\alpha\cos\alpha\sin\theta d\alpha\wedge  d\theta\wedge d\phi,\\
d\Omega_2&=(\sin^2\alpha-\cos^2\alpha)\sin\theta d\alpha\wedge d\phi\wedge d\beta-\sin\alpha\cos\alpha \cos\theta d\theta\wedge d\phi\wedge d\beta,\\
d\Omega_3&=\cos\alpha\cos\theta d\theta\wedge d\alpha\wedge d\phi+(\sin^2\alpha- \cos^2\alpha) d\alpha\wedge d\beta\wedge d\theta.
\end{aligned}
\end{equation*}
We conclude that in these coordinates we have: 
\begin{align*}
2\rho_1=-\cos\alpha d\beta+\cos\theta d\phi; \hspace{10pt} 2 \rho_2=\sin\alpha d\theta; \hspace{10pt} 2\rho_3=\sin\alpha\sin\theta d\phi.
\end{align*}

Now that we have computed the connection forms, we immediately see from (\ref{compVert1Forms}) that the vertical one forms are: 

\begin{equation}\label{Vertical1FormsinFrame SO(3)xId_2}
\begin{aligned}
\xi_0&=d a_0+a_1\left(-\frac{\cos\alpha}{2}d\beta+\frac{\cos\theta}{2}d\phi\right)+a_2 \frac{\sin\alpha}{2} d\theta+a_3\frac{\sin\alpha\sin\theta}{2}d\phi,\\
\xi_1&=d a_1-a_0\left(-\frac{\cos\alpha}{2}d\beta+\frac{\cos\theta}{2}d\phi\right)-a_2\frac{\sin\alpha\sin\theta}{2}d\phi+a_3\frac{\sin\alpha}{2}d\theta,\\
\xi_2&=da_2-a_0\frac{\sin\alpha}{2}d\theta+a_1\frac{\sin\alpha\sin\theta}{2}d\phi-a_3\left(-\frac{\cos\alpha}{2}d\beta+\frac{\cos\theta}{2}d\phi\right),\\
\xi_3&=da_3-a_0\frac{\sin\alpha\sin\theta}{2}d\phi-a_1\frac{\sin\alpha}{2}d\theta+a_2\left(-\frac{\cos\alpha}{2}d\beta+\frac{\cos\theta}{2}d\phi\right).
\end{aligned}
\end{equation}

\subsection{The \texorpdfstring{$\SU(2)$}{Lg} action} \label{Subsection Spin(3) action SO(3)xId_2}
Given the splitting of subsection \ref{Subsection Frame S^4 SO(3)xId_2}, $\R^5= P\oplus P^{\perp}$, since $P\cong\R^3$ and $P^{\perp}\cong \R^2$, we can consider $\SO(3)$ acting in the usual way on $P$ and trivially on $P^{\perp}$. In other words, we see $\SO(3)\cong\SO(P)\times\Id_{P^{\perp}}\subset \SO(P\oplus P^{\perp})\cong \SO(5)$. Obviously, this is also an action on $S^4$.

By taking the differential, $\SO(3)$ acts on the frame bundle $P_{\SO(4)}$ of $S^4$. The theory of covering spaces implies that this action lifts to a $\Spin(3)\cong\SU(2)$ action on the spin structure $P_{\Spin(4)}$ of $S^4$. In particular, the following diagram is commutative: 
\begin{equation}\label{CDgloballifts SO(3)xId_2}
\begin{tikzcd}
& & P_{\Spin(4)} \arrow[dd, "\tilde{\pi}"]\\
& & \\
\Spin(3)\times P_{\Spin(4)} \arrow[uurr, dotted] \arrow[r,"\tilde{\pi}^3_0\times\tilde{\pi}" '] & \SO(3)\times P_{\SO(4)} \arrow[r] & P_{\SO(4)}
\end{tikzcd}.
\end{equation}
Finally, if $\Spin(3)$ acts trivially on $\H$, we can combine the two $\Spin(3)$ actions to obtain one on $P_{\Spin(4)}\times \H$, which descends to the quotient $P_{\Spin(4)}\times_{\mu_{-}} \H=\dS_{\!-}(S^4)$.

\begin{remark}
Recall that $TS^4=P_{\SO(4)}\times_{\boldsymbol{\cdot}}\R^4$, where $\boldsymbol{\cdot}$ is the standard representation of $\SO(4)$ on $\R^4$. Let $G$ be a subgroup of $\SO(5)$ which acts on $P_{\SO(4)}\times_{\boldsymbol{\cdot}}\R^4$ via the differential on the first term and trivially on the second. It is straightforward to verify that this action passes to the quotient and that it coincides with the differential on $TS^4$.
\end{remark}

Now, we describe the geometry of this $\Spin(3)$ action on $\dS_{\!-}(S^4)$. 
Since $\tilde{\pi}$ is fibre-preserving and (\ref{CDgloballifts SO(3)xId_2}) represents a commutative diagram, we observe that, fixed a point $p=(\mathbf{x},\mathbf{y})\in S^4\subset P\oplus P^{\perp}$, the subgroup of $\Spin(3)$ that preserves the fibre of $P_{\Spin(4)}$ over $p$ is the lift of the subgroup of $\SO(3)$ that fixes the fibre of $P_{\SO(4)}$ over $p$.

We first assume $\alpha\neq\pi/2$. The subgroup of $\SO(3)$ that preserves the fibres of $P_{\SO(4)}$ rotates the tangent space of $S^2\subset P$ and fixes the other vectors tangent to $S^4$. Explicitly, if $\{e_i\}_{i=0}^3$ is the oriented orthonormal frame of Subsection \ref{Subsection Frame S^4 SO(3)xId_2} (or an analogous frame when $\alpha=0$, $\theta=0,\pi$), the transformation matrix under the action is:
\begin{align} \label{Action PSO(4) aneqpi/2 SO(3)xId_2}
 h_\gamma:=
\left[
\begin{array}{c|cc}
\Id_2 & & \\
\hline
 & \cos\gamma &-\sin\gamma \\
 & \sin\gamma & \cos\gamma
\end{array}
\right]
\in\SO(4),
\end{align}
for some $\gamma\in [0,2\pi)$.
\begin{claim}\label{Claim1 SO(3)xId_2}
For all $\gamma\in[0,4\pi)$, under the isomorphism $\Spin(4)\cong\Sp(1)\times\Sp(1)$, we have:
\begin{align*}
\tilde{\pi}^4_0 (\tilde{h}_{\gamma})=h_\gamma, 
\end{align*}
where $\tilde{h}_\gamma= (\cos(\gamma/2)+i\sin(\gamma/2),\cos(\gamma/2)+i\sin(\gamma/2))$.
\end{claim}
\begin{proof}
It is well-known that, in this context, $\tilde{\pi}^4_0 \left((l,r)\right) \cdot a=la\overline{r}$ for all $(l,r)\in\Sp(1)\times\Sp(1)$ and all $a\in\H\cong\R^4$. The claim follows from a straightforward computation.
\end{proof}
Using once again the commutativity of (\ref{CDgloballifts SO(3)xId_2}) and Claim \ref{Claim1 SO(3)xId_2}, we deduce that the action in the trivialization of $\dS_{\!-}(S^4)$ induced by $\{e_i\}_{i=0}^3$ is as follows:

\[\begin{tikzcd}[row sep=1ex]
U \times \H \arrow[r, "\cong"]& \left( U\times\Spin(4) \right)\times_{\mu_{-}} \H\arrow[r]& \left( U\times\Spin(4) \right)\times_{\mu_{-}} \H\arrow[r, "\cong"] & U\times \H 
\\
\left(p,a\right) \arrow[r,mapsto]& \left[(p,1_{\Spin(4)}),a \right] \arrow[r,mapsto]& \big[(p,\tilde{h}_{\gamma}),a \big] \arrow[r,mapsto]& \left(p,a\hat{h}_{\gamma}\right)
\end{tikzcd},
\]
where $\hat{h}_{\gamma}:=\cos(\gamma/2)-i\sin(\gamma/2)$ and where $a\in \H$. If we write both $\R^2$ factors of $\H\cong\R^2\oplus\R^2$ in polar coordinates, i.e.,
\[
a=s \cos(\gamma_{-}/2)+i s\sin(\gamma_{-}/2)+j t \cos(\gamma_{+}/2)+k t\sin(\gamma_{+}/2),
\]
for $s,t\in [0,\infty)$ and $\gamma_{\pm}\in [0,4\pi)$, we observe that 
\[
a\hat{h}_{\gamma}=s \cos\left( (\gamma_{-}-\gamma)/2\right)+i s\sin\left( (\gamma_{-}-\gamma)/2\right)+j t \cos\left((\gamma_{+}+\gamma)/2\right)+k t\sin((\gamma_{+}+\gamma)/2).
\]
Geometrically, this is a rotation of angle $-\gamma/2$ on the first $\R^2$ and of angle $\gamma/2$ on the second.

Now, we assume $\alpha=\pi/2$. In this case, the whole $\Spin(3)$ fixes the fibre of $\dS_{\!-}(S^4)$. 
\begin{claim}
$\Spin(3)$ acts on the fibre of $\dS_{\!-}(S^4)$ as $\Sp(1)$ acts on $\H$ via right multiplication of the quaternionic conjugate.
\end{claim}
\begin{proof}
Consider an orthonormal frame such that, at $p=(\underline{0},\cos\beta,\sin\beta)$, has the form:
\begin{align*}
e_0=-\sin\beta\partial_{3}+\cos\beta\partial_{4}; \hspace{15pt}e_1=\partial_0;\hspace{15pt}e_2=\partial_1;\hspace{15pt}e_3=\partial_2,
\end{align*}
where $\partial_i$ are the coordinate vectors of $\R^5\cong P\oplus P^{\perp}$. Observe that the $\SO(3)$ action fixes $e_0$ and acts on $e_1, e_2, e_3$ via matrix multiplication. In particular, given $G\in \SO(3)$, the transformation matrix of the frame at $p$ is:
\[
\left[
\begin{array}{c|c}
1 &  \\
\hline
 & G
\end{array}
\right].
\]
Moreover, for all ${g}\in \Sp(1)\cong \Spin(3)$ and $({g},{g})\in\Sp(1)\times\Sp(1)\cong\Spin(4)$, then
\[
\tilde{\pi}^4_0 ( ({g},{g}))=\left[
\begin{array}{c|c}
1 &  \\
\hline
 & \tilde{\pi}_0^3({g})
\end{array}
\right],
\]
where we recall that $\tilde{\pi}^3_0 (l)\cdot x=l x\overline{l}$ for all $l\in\Sp(1)$ and $x\in \Im\H\cong\R^3$. Indeed, the left-hand side reads: 
\[
\tilde{\pi}_0^4 (({g},{g}))\cdot a={g} a \overline{{g}}={g} \Re a \overline{{g}}+{g} \Im a \overline{{g}}=\Re a+{g} \Im a \overline{{g}},
\]
while the right-hand side is:
\[
\left[
\begin{array}{c|c}
1 &  \\
\hline
 & \tilde{\pi}_0^3({g})
\end{array}\right] a=\left(\begin{array}{c} 
\Re a \\
{g} \Im a\overline{{g}}

\end{array}\right).
\]
We conclude the proof through the commutativity of (\ref{CDgloballifts SO(3)xId_2}).
\end{proof}
We put all these observations in a lemma.
\begin{lemma} \label{Geometry of Spin(3) orbits SO(3)xId_2}
The orbits of the $\SU(2)\cong\Spin(3)$ action on $\dS_{\!-}(S^4)$ are given in Table \ref{Table Orbits SO(3)xId_2}.
\end{lemma}
\begin{table}[h]
\[
\begin{array}{|c|c|c|c|}
\hline
\alpha & (s,t) & \textup{Orbit}\\ \hline
 \neq\frac{\pi}{2} & =(0,0) & S^2 \\ \hline
 \neq\frac{\pi}{2} & \neq(0,0) & S^3 \\ \hline
 =\frac{\pi}{2}& =(0,0) & \textup{Point}\\ \hline
 =\frac{\pi}{2}& \neq(0,0) & S^3 \\ \hline
\end{array}
\] 
\caption{$\Spin(3)$ Orbits}\label{Table Orbits SO(3)xId_2}
\end{table}

\subsection{\texorpdfstring{$\Spin(3)$}{Lg} adapted coordinates}\label{Spin(3) adapted coordinates}
\label{Subsection Spin(3) adapted coordinates SO(3)xId_2}
The description of the $\SU(2)$ action that we carried out in Subsection \ref{Subsection Spin(3) action SO(3)xId_2} suggests the following reparametrization of the linear coordinates $(a_0, a_1, a_2,a_3)$  on the fibres of $\dS_{\!-}(S^4)$:
\begin{align}\label{a_is SO(3)xId_2}
a_0=s\cos\left(\frac{\delta-\gamma}{2}\right); \hspace{10pt} a_1=s\sin\left(\frac{\delta-\gamma}{2}\right); \hspace{10pt} a_2=t\cos\left(\frac{\delta+\gamma}{2}\right); \hspace{10pt} a_3=t\sin\left(\frac{\delta+\gamma}{2}\right),
\end{align}
where $s,t\in [0,\infty)$, $\gamma\in[0,4\pi)$ and $\delta\in[0,2\pi)$. This is a well-defined coordinate system when $s$ and $t$ are strictly positive; we will assume this from now on. Geometrically, $\gamma$ represents the $\SU(2)$ action, while $\delta$ can be either seen as the phase in the orbit of the action when $(a_0,a_1)=(s,0)$ or as twice the common angle in $[0,\pi)$ that the suitable point in the orbit makes with $(s,0)$ and $(t,0)$. These interpretations can be recovered by putting $\gamma=\delta$ and $\gamma=0$, respectively.  

Similarly to \cite{KLo21}, we introduce the standard left-invariant coframe on $\SU(2)$ of coordinates $\gamma,\theta,\phi$ defined on the same intervals as above: 
\begin{align}\label{Coframe orbit SO(3)xId_2}
\sigma_1=d\gamma+\cos\theta d\phi; \hspace{10pt} \sigma_2=\cos\gamma d\theta+\sin\gamma\sin\theta d\phi; \hspace{10pt} \sigma_3=\sin\gamma d\theta-\cos\gamma\sin\theta d\phi.
\end{align}
Observe that: 
\begin{align}\label{sigma_2wedgsigma_3 SO(3)xId_2}
\sigma_2\wedge\sigma_3=-\sin\theta d\theta\wedge d\phi.
\end{align}
Our choice of parametrization of $\dS_{\!-}(S^4)$ implies that (\ref{Coframe orbit SO(3)xId_2}) is a coframe on the $3$-dimensional orbits of the $\SU(2)$ action.

So far, we have constructed a coordinate system $\alpha,\beta,\theta,\phi, s,t, \delta,\gamma$ defining a chart $\mathcal{U}$ of $\dS_{\!-}(S^4)$ and a coframe $\{\sigma_1, \sigma_2,\sigma_3, d\alpha,d\beta,ds,dt,d\delta\}$ on that chart. These coordinates and coframe are such that $\gamma,\theta,\phi$ parametrize the orbits of the $\SU(2)$ action and $\{\sigma_1,\sigma_2,\sigma_3\}$ forms a coframe on these orbits. Let $\{\partial_1, \partial_2,\partial_3, \partial_\alpha,\partial_\beta,\partial_s,\partial_t,\partial_\delta\}$ be the relative dual frame.

\subsection{The \texorpdfstring{$\Spin(7)$}{Lg} geometry in the adapted coordinates} In this subsection, we write the Cayley form $\Phi_c$, as in (\ref{Phi_c}), and the relative metric $g_c$, as in (\ref{g_c}), with respect to the $\SU(2)$ adapted coordinates defined in Subsection \ref{Spin(3) adapted coordinates}.
\begin{lemma} \label{Omegais adapted coordinates SO(3)xId_2}
The horizontal 2-forms $\Omega_1$, $\Omega_2,$ $\Omega_3$, in the adapted frame defined in Subsection \ref{Spin(3) adapted coordinates}, satisfy: 
\begin{align*}
\Omega_1&=\sin\alpha d\alpha\wedge d\beta +\cos^2\alpha \sigma_2\wedge \sigma_3
\end{align*}
and
\begin{align*}
\cos\gamma\Omega_2+\sin\gamma \Omega_3&=\cos\alpha(d\alpha\wedge \sigma_2-\sin\alpha d\beta\wedge \sigma_3),\\
-\sin\gamma\Omega_2+\cos\gamma \Omega_3&=\cos\alpha(-d\alpha\wedge\sigma_3-\sin\alpha  d\beta\wedge \sigma_2).
\end{align*}
\end{lemma}
\begin{proof}
The equations follow from (\ref{Omega_i on S^4 SO(3)xId_2}), (\ref{Coframe orbit SO(3)xId_2}) and (\ref{sigma_2wedgsigma_3 SO(3)xId_2}).
\end{proof}
\begin{lemma} \label{A_i long form SO(3)xId_2}
The vertical 2-forms $A_1$, $A_2,$ $A_3$, in the adapted frame defined in Subsection \ref{Spin(3) adapted coordinates}, have the form: 
\begin{equation*}
\begin{aligned}
A_1=&\frac{1}{2}(sds-tdt)\wedge d\delta+\frac{\cos\alpha}{2}(sds+tdt)\wedge d\beta-\frac{1}{2}(sds+tdt)\wedge\sigma_1\\
&+\frac{\sin\alpha}{2}(tds-sdt)\wedge\sigma_3+(s^2+t^2)\frac{\sin^2\alpha}{4}\sigma_2\wedge\sigma_3+\frac{st\sin\alpha}{2}\sigma_2\wedge d\delta,
\end{aligned}
\end{equation*}

\begin{equation} \label{A_2 and A_3 long form SO(3)xId_2}
\begin{aligned}
A_2=&\cos\gamma ds\wedge dt-\frac{t}{2}\sin\gamma ds\wedge (d\gamma+d\delta)-\frac{s}{2}\sin\gamma dt\wedge (d\delta-d\gamma)-\frac{st}{2}\cos\gamma d\gamma\wedge d\delta\\
&-(s^2+t^2)\frac{\sin\alpha\cos\alpha}{4}\sin\theta d\beta\wedge d\phi +\sin\gamma(sdt-tds)\wedge\left(-\frac{\cos\alpha}{2}d\beta+\frac{\cos\theta}{2}d\phi\right)\\
&+st\cos\gamma d\delta\wedge \left(-\frac{\cos\alpha}{2}d\beta+\frac{\cos\theta}{2}d\phi\right)+\frac{\sin\alpha}{2}d\theta\wedge(tdt+sds)\\
&+\frac{t^2\sin\alpha\sin\theta}{4}(d\gamma+d\delta)\wedge d\phi+\frac{s^2 \sin\alpha\sin\theta}{4}d\phi\wedge (d\delta-d\gamma);\\
A_3=&\sin\gamma ds\wedge dt+\frac{t}{2}\cos\gamma ds\wedge (d\gamma+d\delta)+\frac{s}{2}\cos\gamma dt\wedge (d\delta-d\gamma)+\frac{st}{2}\sin\gamma d\delta\wedge d\gamma\\
&+(s^2+t^2)\frac{\sin\alpha}{4}(\cos\alpha d\beta\wedge d\theta+\cos\theta d\theta\wedge d\phi) -\cos\gamma(sdt-tds)\wedge\left(-\frac{\cos\alpha}{2}d\beta+\frac{\cos\theta}{2}d\phi\right)\\
&+st\sin\gamma d\delta\wedge \left(-\frac{\cos\alpha}{2}d\beta+\frac{\cos\theta}{2}d\phi\right)+\frac{\sin\alpha\sin\theta}{2}d\phi\wedge(tdt+sds)\\
&+\frac{t^2\sin\alpha}{4}d\theta\wedge(d\gamma+d\delta)+\frac{s^2 \sin\alpha}{4} (d\delta-d\gamma)\wedge d\theta.
\end{aligned}
\end{equation}
\end{lemma}
\begin{proof}
Computing the exterior derivatives of the $a_i$'s in the coordinates (\ref{a_is SO(3)xId_2}), we can reduce our statement to a long computation based on (\ref{Vertical1FormsinFrame SO(3)xId_2}).
\end{proof}

\begin{corollary}\label{Ai's in spin(3) coframe SO(3)xId_2}
The vertical 2-forms $A_1$, $A_2,$ $A_3$, in the adapted frame defined in Subsection \ref{Spin(3) adapted coordinates}, satisfy: 
\begin{equation}\label{A_1 nondiag SO(3)xId_2}
\begin{aligned}
A_1=&\left(ds+\frac{t\sin\alpha}{2}\sigma_2\right)\wedge \left(\frac{s}{2}d\delta+\frac{s\cos\alpha}{2}d\beta-\frac{s}{2}\sigma_1+\frac{t\sin\alpha}{2}\sigma_3\right)\\
&- \left(dt-\frac{s\sin\alpha}{2}\sigma_2\right)\wedge \left(\frac{t}{2}d\delta-\frac{t\cos\alpha}{2}d\beta+\frac{t}{2}\sigma_1+\frac{s\sin\alpha}{2}\sigma_3\right)
\end{aligned}
\end{equation}
and 
\begin{align}\label{cosA_2+sinA_3 nondiag SO(3)xId_2}
\cos\gamma A_2+\sin\gamma A_3&=\!\!\!\! \begin{aligned}[t]&\left(ds+\frac{t\sin\alpha}{2}\sigma_2\right)\wedge  \left(dt-\frac{s\sin\alpha}{2}\sigma_2\right)\\
&+\!\!\left(\frac{s}{2}d\delta+\frac{s\cos\alpha}{2}d\beta-\frac{s}{2}\sigma_1+\frac{t\sin\alpha}{2}\sigma_3\right)\!\!\wedge\!\! \left(\frac{t}{2}d\delta-\frac{t\cos\alpha}{2}d\beta+\frac{t}{2}\sigma_1+\frac{s\sin\alpha}{2}\sigma_3\right)\!\!;\end{aligned}\\ \label{cosA_3-sinA_2 nondiag SO(3)xId_2}
\cos\gamma A_3-\sin\gamma A_2 &=\!\!\!\! \begin{aligned}[t] &\left(ds+\frac{t\sin\alpha}{2}\sigma_2\right)\wedge \left(\frac{t}{2}d\delta-\frac{t\cos\alpha}{2}d\beta+\frac{t}{2}\sigma_1+\frac{s\sin\alpha}{2}\sigma_3\right)\\
& +\left(dt-\frac{s\sin\alpha}{2}\sigma_2\right)\wedge \left(\frac{s}{2}d\delta+\frac{s\cos\alpha}{2}d\beta-\frac{s}{2}\sigma_1+\frac{t\sin\alpha}{2}\sigma_3\right).
\end{aligned}
\end{align}
\end{corollary}
\begin{proof}
The first equation in Lemma \ref{A_i long form SO(3)xId_2} is exactly the development of (\ref{A_1 nondiag SO(3)xId_2}). 

A straightforward computation involving (\ref{A_2 and A_3 long form SO(3)xId_2}) gives: 
\begin{align*}
\cos\gamma A_2+\sin\gamma A_3&=\begin{aligned}[t] &ds\wedge dt+\frac{st}{2}d\delta\wedge \sigma_1-\frac{st}{2}\cos\alpha d\delta\wedge d\beta+(s^2+t^2)\frac{\sin\alpha\cos\alpha}{4}d\beta\wedge \sigma_3\\
&+\frac{\sin\alpha}{2}\sigma_2\wedge (tdt+sds)+\frac{(t^2-s^2)\sin\alpha}{4}\sigma_3\wedge d\delta-\frac{(t^2+s^2)\sin\alpha}{4}\sigma_1\wedge\sigma_3;
\end{aligned}\\
\cos\gamma A_3-\sin\gamma A_2&=\begin{aligned}[t] & \frac{1}{2}(tds-sdt)\wedge \sigma_1+\frac{1}{2}(tds+sdt)\wedge d\delta+(s^2+t^2)\frac{\sin\alpha\cos\alpha}{4}d\beta\wedge\sigma_2\\
&+\frac{\cos\alpha}{2}(sdt-tds)\wedge d\beta+\frac{\sin\alpha}{2}(tdt+sds)\wedge \sigma_3
-\frac{(s^2+t^2)\sin\alpha}{4}\sigma_1\wedge \sigma_2\\
&+\frac{(t^2-s^2)\sin\alpha}{4}\sigma_2\wedge d\delta;
\end{aligned}
\end{align*}
which coincide with the development of (\ref{cosA_2+sinA_3 nondiag SO(3)xId_2}) and (\ref{cosA_3-sinA_2 nondiag SO(3)xId_2}), respectively.
\end{proof}

\begin{remark}\label{Phi_c bad coframe SO(3)xId_2}
Using the identities:
\begin{equation}\label{vol=Omega/A_i^2}
\begin{aligned}
b_0\wedge b_1\wedge b_2\wedge b_3&=-\frac{1}{2}\Omega_1\wedge \Omega_1,\\
\xi_0\wedge \xi_1\wedge \xi_2\wedge \xi_3&=-\frac{1}{2}A_1\wedge A_1
\end{aligned}
\end{equation}
and
\begin{equation}\label{sum_i A_iOmega_i}
\begin{aligned}
\sum_{i=1}^3 A_i\wedge \Omega_i= &A_1\wedge \Omega_1+ (\cos\gamma\Omega_2+\sin\gamma\Omega_3)\wedge  (\cos\gamma A_2+\sin\gamma A_3)\\
&+(-\sin\gamma\Omega_2+\cos\gamma\Omega_3)\wedge  (-\sin\gamma A_2+\cos\gamma A_3),
\end{aligned}
\end{equation}
one could easily find $\Phi_c$ in the adapted frame of Subsection \ref{Spin(3) adapted coordinates}. It is clear from Corollary \ref{Ai's in spin(3) coframe SO(3)xId_2} that it is not going to be in a nice form. 
\end{remark}
\begin{lemma}\label{g_c invariant SO(3)xId_2}
Given $c\geq0$, the Riemannian metric $g_c$, in the adapted frame of Subsection \ref{Spin(3) adapted coordinates}, takes the form: 
\begin{equation*} 
\begin{aligned}
g_c=&5(c+r^2)^{3/5}\left(d\alpha^2 +\sin^2\alpha d\beta^2 +\cos^2\alpha(\sigma^2_2+\sigma_3^2)\right) \\
&+4(c+r^2)^{-2/5} \bigg(ds^2+dt^2+\frac{r^2 \cos^2\alpha}{4}d\beta^2+\frac{r^2}{4}\sigma_1^2-\frac{r^2\cos\alpha}{2}d\beta\sigma_1+\frac{r^2\sin^2\alpha}{4}(\sigma_2^2+\sigma_3^2)\\
&+\frac{(t^2-s^2)}{2}d\delta\sigma_1+(st\sin\alpha)d\delta\sigma_3+\frac{r^2}{4} d\delta^2+\sin\alpha(tds-sdt)\sigma_2-\frac{(t^2-s^2)\cos\alpha}{2}d\delta d\beta\!\bigg),
\end{aligned}
\end{equation*}
where $r^2=s^2+t^2$.
\end{lemma}
\begin{proof}
Combining (\ref{g_c}), (\ref{frame S^4 SO(3)xId_2}), (\ref{Vertical1FormsinFrame SO(3)xId_2}) and (\ref{Coframe orbit SO(3)xId_2}), it is easy to obtain the Riemannian metric in the claimed form.
\end{proof}

\subsection{The diagonalizing coframe and frame}\label{Subsection final coframe SO(3)xId_2} In this subsection we define the last coframe on $\dS_{\!-}(S^4)$ that we will use. The motivation comes from the form of $A_1$, $\cos\gamma A_2+\sin\gamma A_3$ and $\cos\gamma A_3-\sin\gamma A_2$ that we obtained in (\ref{A_1 nondiag SO(3)xId_2}), (\ref{cosA_2+sinA_3 nondiag SO(3)xId_2}) and (\ref{cosA_3-sinA_2 nondiag SO(3)xId_2}), respectively. We let: 
\begin{equation} \label{tildedsdtomegai SO(3)xId_2}
\begin{split}
\tilde{ds}&=ds+\frac{t\sin\alpha}{2}\sigma_2;\\
\omega_1&=s d\delta+{s\cos\alpha}d\beta-{s}\sigma_1+{t\sin\alpha}\sigma_3;\\
\end{split}
\quad \quad
\begin{split}
 \tilde{dt}&=dt-\frac{s\sin\alpha}{2}\sigma_2; \\
  \omega_2&=td\delta-{t\cos\alpha}d\beta+{t}\sigma_1+s\sin\alpha\sigma_3.
\end{split}
\end{equation}
Since $t\omega_1+s\omega_2=2tsd\delta+(t^2+s^2)\sin\alpha\sigma_3$ and $s\omega_2-t\omega_1=2st\sigma_1-2st\cos\alpha d\beta+(s^2-t^2)\sin\alpha\sigma_3$, it is clear that $\{\sigma_2, \sigma_3, d\alpha,d\beta, \omega_1,\omega_2,\tilde{ds}, \tilde{dt}\}$ is a coframe on $\mathcal{U}$. Let $\{e_2, e_3, e_\alpha, e_\beta, e_{\omega_1}, e_{\omega_2},e_s, e_t\}$ denote the relative dual frame.
\begin{corollary}\label{Ai's in diagonalizing coframe SO(3)xId_2}
The vertical 2-forms $A_1$, $A_2,$ $A_3$, in the coframe defined in this subsection, satisfy: 
\begin{equation}
\begin{aligned}
A_1=\frac{1}{2}\left(\tilde{ds}\wedge \omega_1- \tilde{dt} \wedge \omega_2\right)
\end{aligned}
\end{equation}
and 
\begin{align}
\cos\gamma A_2+\sin\gamma A_3&= \tilde{ds}\wedge \tilde{dt}+\frac{1}{4}\omega_1\wedge\omega_2;\\
\cos\gamma A_3-\sin\gamma A_2 &=\frac{1}{2}\left(\tilde{ds}\wedge \omega_2+\tilde{dt}\wedge \omega_1 \right).
\end{align}
\end{corollary}
\begin{proof}
It follows immediately from Corollary \ref{Ai's in spin(3) coframe SO(3)xId_2} and (\ref{tildedsdtomegai SO(3)xId_2}).
\end{proof}

\begin{proposition}
Given $c\geq0$, the Cayley form $\Phi_c$, in the coframe defined in this subsection, satisfies: 
\begin{equation}\label{Phi_c final coordinates SO(3)xId_2}
\begin{aligned}
\Phi_c=&4(c+r^2)^{-4/5} \tilde{ds}\wedge\tilde{dt}\wedge\omega_2\wedge\omega_1 +25(c+r^2)^{6/5}\sin\alpha\cos^2\alpha d\alpha\wedge d\beta\wedge\sigma_3\wedge\sigma_2\\
&10(c+r^2)^{1/5} \bigg( \left(\tilde{ds}\wedge \omega_1- \tilde{dt} \wedge \omega_2\right)\wedge\left(\sin\alpha d\alpha\wedge d\beta +\cos^2\alpha \sigma_2\wedge \sigma_3\right)\\
&+\frac{1}{2}\left( 4\tilde{ds}\wedge \tilde{dt}+\omega_1\wedge\omega_2\right)\wedge\left(\cos\alpha(d\alpha\wedge \sigma_2-\sin\alpha d\beta\wedge \sigma_3)\right)\\
&+\left(\tilde{ds}\wedge \omega_2+\tilde{dt}\wedge \omega_1 \right)\wedge \cos\alpha\left(-d\alpha\wedge\sigma_3-\sin\alpha  d\beta\wedge \sigma_2\right)\bigg),
\end{aligned}
\end{equation}
where $r^2=s^2+t^2$.
\end{proposition}

\begin{proof}
This is a straightforward consequence of Lemma \ref{Omegais adapted coordinates SO(3)xId_2}, (\ref{vol=Omega/A_i^2}), (\ref{sum_i A_iOmega_i}) and Corollary \ref{Ai's in diagonalizing coframe SO(3)xId_2}.
\end{proof}

\begin{proposition} \label{Proposition Diagonal metric SO(3)xId_2}
Given $c\geq0$, the Riemannian metric $g_c$, in the coframe defined in this subsection, satisfies: 
\begin{align}\label{metric Diagonal SO(3)xId_2}
g_c=5(c+r^2)^{3/5}\left( d\alpha^2+\sin^2\alpha d\beta^2+\cos^2\alpha \left(\sigma_2^2+\sigma_3^2\right)\right)+4(c+r^2)^{-2/5} \left(\tilde{ds}^2+\tilde{dt}^2+\frac{\left(\omega_1^2+\omega_2^2\right)}{4}\right),
\end{align}
where $r^2=s^2+t^2$.
\end{proposition}
\begin{proof}
The first addendum remains invariant from Lemma \ref{g_c invariant SO(3)xId_2}, while (\ref{tildedsdtomegai SO(3)xId_2}) implies that the remaining part is equal to the second addendum in Lemma \ref{g_c invariant SO(3)xId_2}.
\end{proof}
In particular, using this coframe, we sacrifice compatibility with the group action to obtain a simpler form for $\Phi_c$ and a diagonal metric.

We conclude this subsection by computing the dual frame with respect to the $\SU(2)$ adapted frame $\{\partial_1, \partial_2,\partial_3, \partial_\alpha,\partial_\beta,\partial_s,\partial_t,\partial_\delta\}$.
\begin{lemma} \label{e_i w.r.t partial_i SO(3)xId_2}
The dual frame $\{e_2, e_3, e_\alpha, e_\beta, e_{\omega_1}, e_{\omega_2},e_s, e_t\}$ satisfies:
\begin{equation} 
\begin{split}
e_\alpha&=\partial_\alpha;\\
e_2&=\partial_2 -\frac{t\sin\alpha}{2}\partial_s+\frac{s\sin\alpha}{2}\partial_t; \\
e_{s}&=\partial_s; \\
e_{\omega_1}&=\frac{1}{2s}\partial_{\delta}-\frac{1}{2s}\partial_1;
\end{split}
\quad\quad
\begin{split}
e_\beta&=\partial_\beta+\cos\alpha\partial_1;\\
e_3&=\partial_3-\frac{(s^2+t^2)\sin\alpha}{2st}\partial_{\delta}+\frac{(t^2-s^2)\sin\alpha}{2st}\partial_1;\\
e_t&=\partial_t;\\
e_{\omega_2}&=\frac{1}{2t}\partial_{\delta}+\frac{1}{2t}\partial_1;
\end{split}
\end{equation}
where $\{\partial_1, \partial_2,\partial_3, \partial_\alpha,\partial_\beta,\partial_s,\partial_t,\partial_\delta\}$ is the dual frame with respect to the $\SU(2)$ adapted coordinates of Subsection \ref{Subsection Spin(3) adapted coordinates SO(3)xId_2}.
\end{lemma}
\begin{proof}
It is straightforward to verify these identities from (\ref{tildedsdtomegai SO(3)xId_2}) and the definition of dual frame. 
\end{proof}

\subsection{The Cayley condition}\label{subsection Cayley condition SO(3)xId_2} 
As the generic orbit of the $\SU(2)$ action we are considering is $3$-dimensional (see Lemma \ref{Geometry of Spin(3) orbits SO(3)xId_2}), it is sensible to look for $\SU(2)$-invariant Cayley submanifolds. Indeed, Harvey and Lawson theorem \cite[Theorem IV .4.3]{HL82} guarantees the local existence and uniqueness of a Cayley passing through any given generic orbit. To construct such a submanifold $N$, we consider a 1-parameter family of $3$-dimensional $\SU(2)$-orbits in $M$. Hence, the coordinates that do not describe the orbits, i.e. $\alpha$, $\beta$, $s$, $t$ and $\delta$, need to be functions of a parameter $\tau$. Explicitly, we have:
\begin{equation}
\begin{aligned}
N=&\bigg\{\bigg( (\cos\alpha(\tau)\textbf{u},\sin\alpha(\tau)\textbf{v}),\bigg((s(\tau)\cos\left(\frac{\delta(\tau)-\gamma}{2}\right),s(\tau)\sin\left(\frac{\delta(\tau)-\gamma}{2}\right),\\
&t(\tau)\cos\left( \frac{\delta(\tau)+\gamma}{2}\right), t(\tau)\sin\left(\frac{\delta(\tau)+\gamma}{2}\right)\bigg)\bigg): \av{\textbf{u}}=1,\av{\textbf{v}}=1, \gamma\in [0,4\pi),\tau\in(-\epsilon,\epsilon)\bigg\},
\end{aligned}
\end{equation}
and its tangent space is spanned by: $\{\partial_1,\partial_2,\partial_3, \dot{s}\partial_s+\dot{t}\partial_t+\dot{\alpha}\partial_\alpha+\dot{\beta}\partial_\beta+\dot{\delta}\partial_{\delta}\},$ where the dots denotes the derivative with respect to $\tau$. The Cayley condition imposed on this tangent space (see Proposition \ref{Cayley condition KM}) generates a system of ODEs on $\alpha,\beta,s,t,\delta$.

\begin{theorem} \label{Complicated ODEs SO(3)xId_2}
Let $N$ be an $\SU(2)$-invariant submanifold as described at the beginning of this subsection. Then, $N$ is Cayley in the chart $\mathcal{U}$ if and only if the following system of ODEs is satisfied: 
\begin{equation}\label{complicated system SO(3)xId_2}
\left\{
\begin{aligned}
&(s^2+t^2)\sin^2\alpha\cos\alpha \dot\beta=0\\
&\cos^2\alpha(t\dot{s}-s\dot{t})=0\\
&\cos^2\alpha st\dot{\delta}=0\\
&-5(c+r^2)\cos^2\alpha s\dot{\alpha}+r^2 \sin^2\alpha\dot{\alpha} s-2\sin\alpha\cos\alpha t^2 \dot{s} -4\cos\alpha\sin\alpha s^2\dot{s}-2\sin\alpha\cos\alpha st\dot{t}=0\\
&5(c+r^2)\cos^2\alpha t\dot{\alpha}-r^2 \sin^2\alpha\dot{\alpha} t+2\sin\alpha\cos\alpha s^2 \dot{t} +4\cos\alpha\sin\alpha t^2\dot{t}+2\sin\alpha\cos\alpha st\dot{s}=0\\
&5(c+r^2)\sin\alpha\cos^2\alpha \dot{\beta}s-2\sin\alpha\cos\alpha t^2 s\dot{\delta}-r^2 \sin^3\alpha \dot{\beta}s=0\\
&-5(c+r^2)\sin\alpha\cos^2\alpha \dot{\beta}t-2\sin\alpha\cos\alpha ts^2 \dot{\delta}+r^2 \sin^3\alpha \dot{\beta}t=0
\end{aligned}
\right. ,
\end{equation}
where $r^2=s^2+t^2$ as usual.
\end{theorem}
As it mainly consists of computations, we leave the proof of Theorem \ref{Complicated ODEs SO(3)xId_2} to Appendix \ref{appendix proof of Theorem ODES SO(3)xId_2}. 

\begin{corollary}\label{easy ODEs SO(3)xId_2}
Let $N$ be an $\SU(2)$-invariant submanifold as described at the beginning of this subsection. Then, $N$ is Cayley in the chart $\mathcal{U}$ if and only if the following system of ODEs is satisfied: 
\begin{equation*}
\left\{
\begin{aligned}
&\dot\beta=0\\
&(t\dot{s}-s\dot{t})=0\\
&\dot{\delta}=0\\
&5(c+r^2)\cos^2\alpha st\dot{\alpha}-(s^2+t^2)st \sin^2\alpha\dot{\alpha} +2\sin\alpha\cos\alpha (s^2+t^2) (s\dot{t}+t\dot{s})=0\\
\end{aligned}
\right. ,
\end{equation*}
where $r^2=s^2+t^2$ as usual.
\end{corollary}
\begin{proof}
 As $\alpha\in(0,\pi/2)$ and $s,t>0$, we get immediately the first three equations from the first three equations of (\ref{complicated system SO(3)xId_2}). The last two equations of (\ref{complicated system SO(3)xId_2}) are superfluous as $\dot{\beta}=0$ and $\dot{\delta}=0$. The same holds for  $t$ times the fourth equation plus $s$ times the fifth equation of (\ref{complicated system SO(3)xId_2}), where we use $t\dot{s}-s\dot{t}=0$ this time. We conclude by considering $s$ times the fifth equation minus $t$ times the fourth equation of (\ref{complicated system SO(3)xId_2}).
\end{proof}

\subsection{The Cayley fibration} In the previous section we found the condition that makes $N$, $\SU(2)$-invariant submanifold, a Cayley submanifold. Explicitly, it consists of a system of ODEs that is completely integrable; these solutions will give us the desired fibration. 
\begin{theorem}
Let $N$ be an $\SU(2)$-invariant submanifold as described at the beginning of Subsection \ref{subsection Cayley condition SO(3)xId_2}. Then, $N$ is Cayley in $\mathcal{U}$ if and only if the following quantities are constant: 
\begin{align*}
\beta,\hspace{30pt} \delta,\hspace{30pt}  \frac{s}{t},\hspace{30pt} F:= 2\sin^{5/2} \alpha\cos^{1/2} \alpha st+5c\frac{st}{(s^2+t^2)}H(\alpha), 
\end{align*}
where $H(\alpha)$ is the primitive function of $h(\alpha):=(\cos\alpha\sin\alpha)^{3/2}$.
\end{theorem}
\begin{proof}
The condition on $\beta$ and $\delta$ follows immediately from Corollary \ref{easy ODEs SO(3)xId_2}. Taking the derivative in $\tau$ of $s/t$, we see that
\begin{align*}
0=\frac{d}{d\tau}\left(\frac{s}{t}\right)=\frac{\dot{s}t-\dot{t}s}{t^2}, 
\end{align*}
which is equivalent to the second equation in Corollary \ref{easy ODEs SO(3)xId_2}, as $t>0$. Analogously, one can see that the derivative with respect to $\tau$ of $F$ is equivalent to the last equation of Corollary \ref{easy ODEs SO(3)xId_2} if we assume that $s/t$ is constant.
\end{proof}

Setting 
$$v:=\frac{s}{t}, \hspace{25pt} u:=st,$$ the preserved quantities transform to:
\begin{align*}
\beta,\hspace{30pt} \delta,\hspace{30pt} v,\hspace{30pt} F:= 2\sin^{5/2} \alpha\cos^{1/2} \alpha(v^2+1)u+5cvH(\alpha), 
\end{align*}
where we multiplied $F$ by the constant $(v^2+1)$.
Observe that this is an admissible transformation from $s,t\in (0,\infty)$ to $u,v\in(0,\infty)$. Moreover, fixed $\beta,\delta,v$, we can represent the $\SU(2)$-invariant Cayley submanifolds as the level sets of $F$ reckoned as a $\R$-valued function of $\alpha$ and $u$.  An easy analysis of $F$ shows that these level sets can be represented as in Figure \ref{Figure of Cayley fibration SO(3)xId_2}. 
\begin{figure}
    \centering
   \begin{subfigure}{0.48\textwidth}\centering\includegraphics[scale=0.37]{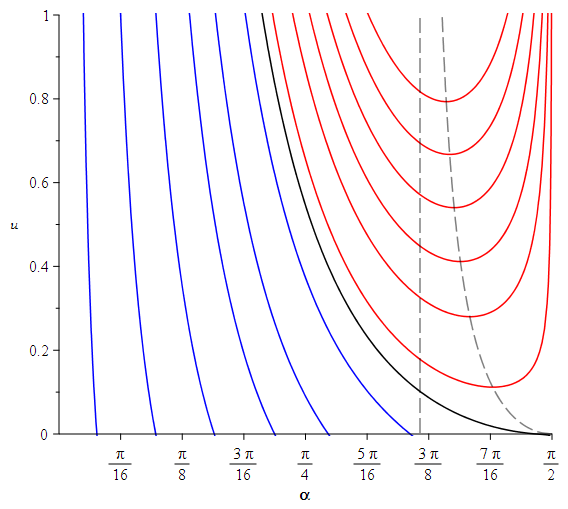}\caption{Level sets of F with c=1 and v=1}%
    \end{subfigure}
    \begin{subfigure}{0.48\textwidth}\centering\includegraphics[scale=0.37]{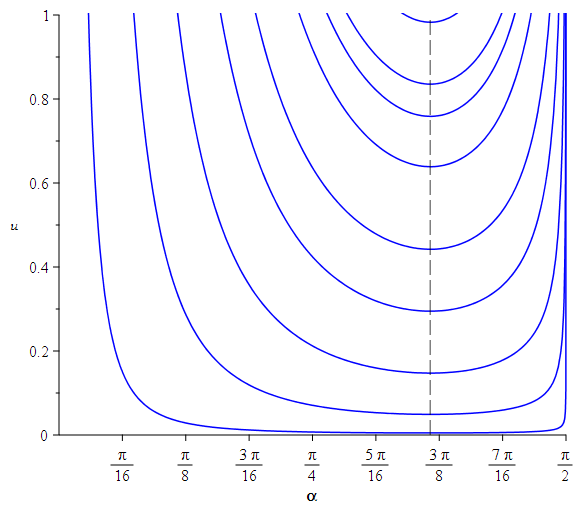}\caption{Level sets of F with c=0 and v=1}
    \end{subfigure}
    \caption{Level sets of $F$ in the generic and in the conical case}%
    \label{Figure of Cayley fibration SO(3)xId_2}%
\end{figure}
The dashed lines in the two graphs correspond to the curves formed by the $u$-minimums of each level set and to the two vertical lines: $\alpha=\arccos(1/\sqrt{6})$. For $c=0$, these coincide, while in the generic case the locus of the $u$-minimum is:
\[
\alpha=\arccos\left(\sqrt{\frac{u(v^2+1)}{6u(v^2+1)+5cv}}\right),
\]
which is only asymptotic to $\alpha=\arccos(1/\sqrt{6})$ for $u\to\infty$.

\subsection*{The conical version.} We first consider the easier case, i.e. when $c=0$. It is clear from the graph that the $\SU(2)$-invariant Cayleys passing through $\mathcal{U}$ are  contained in $\mathcal{U}$, have topology $S^3\times \R$ and are smooth. Moreover, we can construct a Cayley fibration on the chart $\mathcal{U}$ with base an open subset of $\R^4$. To do so, we associate to each point of $\mathcal{U}$ the value of $\beta$, $\delta$, $s/t$ and $F$ of the Cayley passing through that point. This $\SU(2)$-invariant fibration naturally extends to the whole $M_0$ via continuity. Using Table \ref{Table Orbits SO(3)xId_2} and Harvey and Lawson uniqueness theorem \cite[Theorem IV.4.3]{HL82}, we can describe the extension precisely. Indeed, when $\alpha=\pi/2$, the fibres of $\pi_{S^4}$ are $\SU(2)$-invariant Cayley submanifolds; when $\alpha\neq\pi/2$ and $s=0$ or $t=0$, the suitable Cayley submanifolds constructed by Karigiannis and Min-Oo \cite{KM05} are $\SU(2)$-invariant; finally, when $\alpha=0$ and $(s,t)\neq 0$, the fibres are given by an extension of \cite{KLe12}. The topology of these Cayley submanifolds that are not contained in $\mathcal{U}$ is $\R^4\setminus \{0\}$ in the first case and $\R\times S^3$ in the remaining ones. Observe that this fibration does not admit singular or intersecting fibres.

\subsection*{The smooth version.}  Now, we consider the generic case, i.e. when $c>0$. Differently from the cone, the graph of the level sets of $F$ shows that the $\SU(2)$-invariant Cayley submanifolds passing through $\mathcal{U}$ do not remain contained in it, and they admit three different topologies in the extension. The red, black and blue lines correspond to submanifolds with topology $\R\times S^3$, $\R^4$ and $\mathcal{O}_{\C\P^1} (-1)$, respectively. 
Indeed, the first two cases are obvious, while, if we assume smoothness, we can deduce the third one through the argument of Subsection \ref{subsection cayley submanifolds and cayley fibration}. We define an $\SU(2)$-invariant Cayley fibration on $\mathcal{U}$ that extends to the whole $M$ exactly as above. If we fix a value of $F$ corresponding to a Cayley of topology $\mathcal{O}_{\C\P^1} (-1)$, then, for every $\delta$, $v$, all the different Cayleys will intersect in a $\C\P^1\subset S^4$, where $S^4$ is the zero section of $\dS_{\!-}(S^4)$. In particular, the $M'$ of Definition \ref{Cayley fibration definition} is equal to $M_0$ in this context, i.e., we can assume $u>0$.

\subsection*{The parametrizing space.} Using Figure \ref{Figure of Cayley fibration SO(3)xId_2}, we can study the parametrizing space $\mathcal{B}$ of the Cayley fibrations we have just described. We will only deal with the smooth version, as the conical case is going to be completely analogous. 

Ignoring $\beta$ for a moment, it is immediate to see that, if we restrict our attention to the fibres that are topologically $\mathcal{O}_{\C\P^1}{(-1)}$ and the ones corresponding to the black line, the parametrizing space is homeomorphic to $S^2\times [0,1]$. The remaining fibres are parametrized by $B^3 (1)$, open unit ball of $\R^3$. As we removed the zero section of $\dS_{\!-}(S^4)$, it is clear that we can glue these partial parametrizations together to obtain $\overline{B^3(2)}$. Now, $\beta$ gives a circle action on $\overline{B^3(2)}$ that vanishes on its boundary. We conclude that the parametrizing space $\mathcal{B}$ of the smooth Cayley fibration is $S^4$. Indeed, this is essentially the same way to describe $S^4$ as we did in Subsection \ref{Subsection Frame S^4 SO(3)xId_2}. 

\subsection*{The smoothness of the fibres (the asymptotic analysis).} In this subsection, we study the smoothness of the fibres. Observe that this property is obviously satisfied as long as they are contained in the chart $\mathcal{U}$. Hence, the Cayleys of topology $S^3\times \R$ are smooth, and we only need to check the remaining ones in the points where they meet the zero section, i.e., when the $\SU(2)$ group action degenerates. To this purpose, we carry out an asymptotic analysis. 

Let $\beta_0,v_0,\delta_0$ and $F_0$ be the constants determining a Cayley fibre $N$. By the explicit formula for $F$, we see that $N$ is given by:
\begin{align*}
	u=\frac{F_0-5cv_0H(\alpha)}{2\sin^{5/2}\alpha\cos^{1/2}\alpha(v_0^2+1)}
\end{align*}

We first check the smoothness of the fibres that meet the zero section ($u=0$) at some $\alpha_0\in(0,\pi/2)$, i.e., the ones of topology $\mathcal{O}_{\C\P^1}(-1)$. For this purpose, if we expand near $\alpha_0^-$ and we obtain the linear approximation of $N$ at that point. Explicitly, this is the $\SU(2)$-invariant $4$-dimensional submanifold $\Sigma$ characterized by the equation
\[
u=-\frac{5cv_0}{2\tan\alpha_0(v_0^2+1)}(\alpha-\alpha_0),
\]
and where $v,\delta, \beta$ are constantly equal to $v_0,\delta_0,\beta_0$.

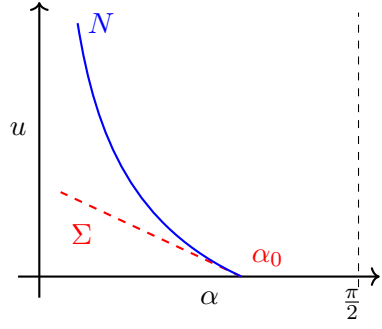
\begin{figure}
\begin{tikzpicture}[scale=0.28]

\draw[->,thick] (-1,0)--(16,0);
\draw[->,thick] (0,-1)--(0,13); 
\draw[-,dashed] (15,-0.5)--(15,12.5); 
\draw[red,thick, dashed] (9.5,0)--(1,4);

\draw[blue, thick] (1.8,12) to  [out=279,in=155] (9.5,0);

\node at (8,-1) [black]{$\alpha$};
\node at (-1,7) [black]{$u$};
\node at (14.7,-1.3) [black]{$\frac{\pi}{2}$};
\node at (9.5,0) [above right,red]{$\alpha_0$};
\node at (1.8,12) [right,blue]{$N$};
\node at (2,2) [red]{$\Sigma$};
\end{tikzpicture}
\caption{Approximation of a Cayley at $u=0$ when $\alpha_0\in(0,\pi/2)$}\label{Figure linear approximation SO(3)xId_2}
\end{figure}

Now, we want to study the asymptotic behaviour of the metric $g_c$ when restricted to $\Sigma$, and then, we let $\alpha$ tends to $\alpha_0$ from the left. To do so, it is convenient to compute the following identities using the definition of $u:=st$ and $v:=s/t$:
\begin{equation}\label{transformation ds du SO(3)xId_2}
\begin{aligned}
dt&=\frac{1}{2\sqrt{uv}}du-\frac{1}{2v}\sqrt{\frac{u}{v}}dv,\\
ds&=\frac{\sqrt{v}}{2\sqrt{u}}du+\frac{\sqrt{u}}{2\sqrt{v}}dv,\\
ds^2&=\frac{v}{4u}du^2+\frac{u}{4v}dv^2+\frac{1}{2}dudv,\\
dt^2&=\frac{1}{4uv}du^2 +\frac{u}{4v^3}dv^2-\frac{1}{2v^2}dudv.\\
\end{aligned}
\end{equation}
The metric $g_c$, in the coframe $\{\sigma_1, \sigma_2,\sigma_3, d\alpha,d\beta,du,dv,d\delta\}$, then can be rewritten as:
\begin{equation} \label{metric for asymptotic/singular geometry SO(3)xId_2} 
\begin{aligned}
g_c=&5\left(c+ \frac{u}{v}(1+v^2)\right)^{3/5}\left(d\alpha^2 +\sin^2\alpha d\beta^2 +\cos^2\alpha(\sigma^2_2+\sigma_3^2)\right) \\
&+4\left(c+ \frac{u}{v}(1+v^2)\right)^{-2/5} \bigg(\frac{1}{4uv}(1+v^2)du^2+\frac{u}{4v^3}(1+v^2)dv^2+\frac{1}{2v^2}(v^2-1)dudv \\
&+ \frac{u}{v}(1+v^2)\frac{\cos^2\alpha}{4}d\beta^2+ \frac{u}{4v}(1+v^2)\sigma_1^2-\frac{\cos\alpha}{2} \frac{u}{v}(1+v^2)d\beta\sigma_1+ \frac{u}{v}(1+v^2)\frac{\sin^2\alpha}{4}(\sigma_2^2+\sigma_3^2)\\
&+\frac{u(1-v^2)}{2v}d\delta\sigma_1+u\sin\alpha d\delta\sigma_3+ \frac{u}{4v}(1+v^2) d\delta^2+\sin\alpha\frac{u}{v}dv\sigma_2-\frac{u(1-v^2)\cos\alpha}{2v}d\delta d\beta\!\bigg),
\end{aligned}
\end{equation}
where we used (\ref{transformation ds du SO(3)xId_2}) and Lemma \ref{g_c invariant SO(3)xId_2}. Now, if we restrict (\ref{metric for asymptotic/singular geometry SO(3)xId_2}) to $\Sigma$, and we let $\alpha$ tend to $\alpha_0$ from the left, we get: 
\begin{align*}
\restr{g_c}{N}&\sim \frac{c^{-2/5}}{v_0}(1+v_0^2)\left(\frac{du^2}{u}+u\sigma_1^2\right)+5c^{3/5} \cos^2\alpha_0 (\sigma_2^2+\sigma_3^2)\\
&\sim dr^2+r^2 \frac{\sigma_1^2}{4}+5c^{3/5} \cos^2\alpha_0 (\sigma_2^2+\sigma_3^2),
\end{align*}
where 
\[
r=\sqrt{\frac{1+v_0^2}{v_0 c^{2/5}}} 2\sqrt{u}.
\]
As the length of $\sigma_1$ is $4\pi$, we deduce that the metric $g_c$ extends smoothly to the $\C\P^1 \cong S^2$ contained in the zero section. This two-dimensional sphere corresponds to the base of the bundle $\mathcal{O}_{\C\P^1}(-1)$.

Finally, we check the smoothness of the fibres meeting the zero section at $\alpha_0=\pi/2$, i.e., the ones with topology $\R^4$. Expanding for $\alpha\to\pi/2^-$, we immediately see that the first order is not enough and we need to pass to second order. Explicitly, this is the $\SU(2)$-invariant $4$-dimensional submanifold $\Sigma$ of equation:
\[
u=A(\alpha-\pi/2)^2,
\]
where $A:={cv}({1+v^2})^{-1}$ is the constant depending on $c,v$ determined by the expansion. As above, the remaining parameters $v,\delta, \beta$ are constantly equal to $v_0,\delta_0,\beta_0$.
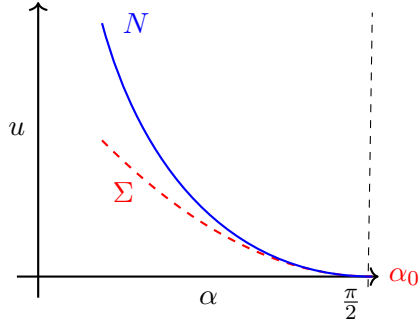
\begin{figure}
\begin{tikzpicture}[scale=0.28]
\draw[->,thick] (-1,0)--(16,0);
\draw[->,thick] (0,-1)--(0,13); 
\draw[-,dashed] (15.5,-0.5)--(15.7,12.5); 
\draw [red, thick, dashed, domain=3:15.7] plot (\x, {0.04*(\x-15.7)^2});
\draw[blue, thick] (3,12) to  [out=285,in=180] (15.7,0);
\node at (8,-1) [black]{$\alpha$};
\node at (-1,7) [black]{$u$};
\node at (14.7,-1.3) [black]{$\frac{\pi}{2}$};
\node at (16,0) [right,red]{$\alpha_0$};
\node at (3.5,12) [right,blue]{$N$};
\node at (4,4) [red]{$\Sigma$};
\end{tikzpicture}
\caption{Approximation of a Cayley at $u=0$ when $\alpha_0=\pi/2$}\label{Figure quadratic approximation SO(3)xId_2}
\end{figure}
If we restrict $g_c$ as defined in $(\ref{metric for asymptotic/singular geometry SO(3)xId_2})$ to $\Sigma$, and we let $\alpha$ tend to $\pi/2$, then, we obtain: 
\begin{equation*}
\begin{aligned}
\restr{g_c}{N}&\sim\begin{aligned}[t] &5c^{3/5} (\alpha-\pi/2)^2 (\sigma_2^2+\sigma_3^2)+A c^{-2/5}\left(\frac{1+v^2}{v}\right)(\alpha-\pi/2)^2 (\sigma_1^2+\sigma_2^2+\sigma_3^2)\\
&+\left(5c^{3/5}+4Ac^{-2/5}\frac{1+v^2}{v}\right)d\alpha^2\end{aligned}\\
&\sim c^{3/5} (\alpha-\alpha_0)^2 \left(\sigma_1^2+6(\sigma_2^2+\sigma_3^2)\right)+9c^{3/5}d\alpha^2,
\end{aligned}
\end{equation*}
where we also used the expansion of $\cos\alpha$ around $\pi/2$ and the explicit value of $A$. We conclude that $N$ is not smooth when it meets the zero section, and it develops an asymptotically conical singularity at that point. 

\begin{remark}
The singularity is asymptotic to the Lawson--Osserman cone \cite{LO77}.
\end{remark}

\subsection*{The main theorems} We collect all these results in the following theorems. Observe that we are using the notion of Cayley fibration given in Definition $\ref{Cayley fibration definition}$.
\begin{theorem}[Generic case]\label{Main theorem Cayley fibration generic SO(3)xId_2}
Let $(M,\Phi_c)$ be the Bryant--Salamon manifold constructed over the round sphere $S^4$ for some $c>0$, and let $\SU(2)$ act on $M$ as in Subsection \ref{Subsection Spin(3) action SO(3)xId_2}. Then, $M$ admits an $\SU(2)$-invariant Cayley fibration parametrized by $\mathcal{B}\cong S^4$. The fibres are topologically $\mathcal{O}_{\C\P^1} (-1)$, $S^3\times \R$ and $\R^4$. Apart from the non-vertical fibres of topology $\R^4$, all the others are smooth. The singular fibres of the Cayley fibration have a conically singular point and are parametrized by $(\mathcal{B}^{\circ})^c \cong S^2\times S^1$ ($\beta,\delta, v$ in our description). Moreover, at each point of the zero section $S^4\subset\dS_{\!-}(S^4)$, infinitely many Cayley fibres intersect.
\end{theorem}
\begin{theorem}[Conical case]\label{Main theorem Cayley fibration conical SO(3)xId_2}
Let $(M_0,\Phi_0)$ be the conical Bryant--Salamon manifold constructed over the round sphere $S^4$, and let $\SU(2)$ act on $M_0$ as in Subsection \ref{Subsection Spin(3) action SO(3)xId_2}. Then, $M_0$ admits an $\SU(2)$-invariant Cayley fibration parametrized by $\mathcal{B}\cong S^4$. The fibres are topologically $S^3\times\R$ and are all smooth. Moreover, as these do not intersect, the $\SU(2)$-invariant Cayley fibration is a fibration in the usual differential geometric sense with fibres Cayley submanifolds. 
\end{theorem}

\begin{remark}
It is interesting to observe that, in the generic case, the family of singular $\R^4$s separates the fibres of topology $S^3\times \R$ from the ones of topology $\mathcal{O}_{\C\P^1}(-1)$. 
\end{remark}

\begin{remark}
Similarly to \cite[Subsection 5.11.1]{KLo21}, one can blow-up at the north pole and argue that in the limit the Cayley fibration splits into the product of a line $\R$ and of an $\SU(2)$-invariant coassociative fibration on $\R^7$. By the uniqueness of the $\SU(2)$-invariant coassociative fibrations of $\R^7$, we deduce that the latter is the Harvey and Lawson coassociative fibration \cite[Section IV.3]{HL82} up to a reparametrization.
\end{remark}

\begin{remark}
	From the computations that we have carried out, it is easy to give an explicit formula for the multi-moment map $\nu_c$ associated to this action. Indeed, this is:\begin{align*}
\nu_c=&5(c+s^2+t^2)^{1/5} \left((s^2+t^2)\cos^2\alpha-\frac{1}{6}(s^2+t^2-5c)\right)-\frac{25}{6}c^{6/5} \hspace{15pt} c\geq 0.
\end{align*}
Obviously, the range of $\nu_c$ is the whole $\R$.
  Under the usual transformation $u=st$ and $v=s/t$, the multi-moment map becomes: 
\begin{align*}
\nu_c=&\frac{5}{6}\left(c+\frac{u(1+v^2)}{v}\right)^{1/5} \left(6\frac{u(1+v^2)}{v}\cos^2\alpha-\frac{u(1+v^2)}{v}+5c\right)-\frac{25}{6}c^{6/5}.
\end{align*}
We draw the level sets of $\nu_c$ in Figure \ref{multi-moment map SO(3)xId_2}.

The black lines correspond to the level set relative to zero, the red lines correspond to negative values, while the blue lines correspond to the positive ones.

Differently from the conical case, the $0$-level set of $\nu_c$ for $c>0$ does not coincide with the locus of $u$-minimum of each level set of $F$. Moreover, for every $c\geq0$, it does not even coincide with the set of $\SU(2)$-orbits of minimum volume in each fibre.
\end{remark}

\begin{figure}
    \centering
    \begin{subfigure}{0.48\textwidth}\centering\includegraphics[scale=0.35]{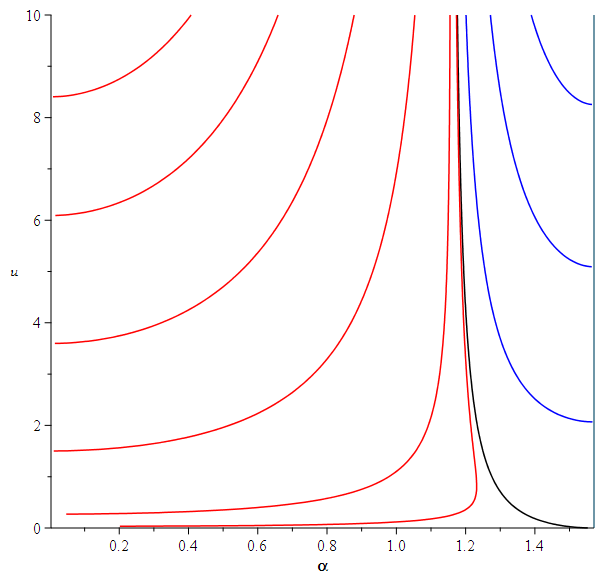} \caption{Level sets of $\nu_1$ with $v=1$}
    \end{subfigure}
    \begin{subfigure}{0.48\textwidth}\centering\includegraphics[scale=0.35]{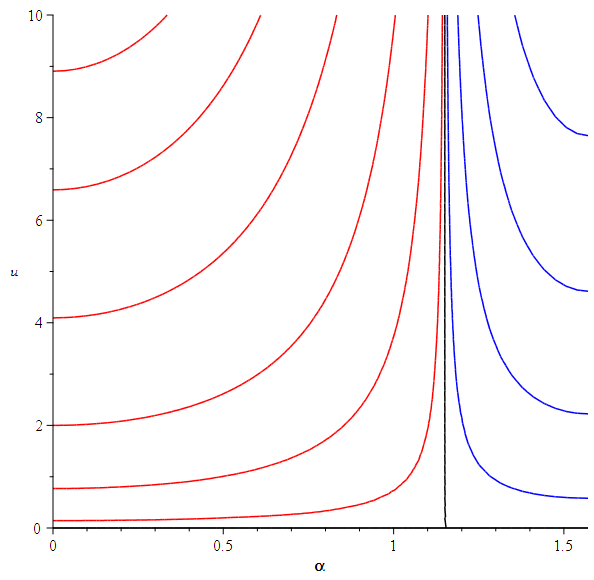}\caption{Level sets of $\nu_0$ with $v=1$}
    \end{subfigure}
    \caption{Level sets of the multi-moment map in the generic and conical case}%
    \label{multi-moment map SO(3)xId_2}%
\end{figure}

\subsection*{Asymptotic geometry.} Inspecting the geometry of the Cayley fibration (see Figure \ref{Figure of Cayley fibration SO(3)xId_2}), we deduce that there are two asymptotic behaviours for the fibres: one for $\alpha\sim0$ and one for $\alpha\sim\pi/2$. In both cases, as $u\to\infty$, the tangent space of the Cayley fibre $N$ tends to be spanned by $\partial_u, \partial_1,\partial_2,\partial_3$. We can use the formula for the metric (\ref{metric for asymptotic/singular geometry SO(3)xId_2}) to obtain, for $\alpha\sim0$: 
\begin{align*}
\restr{g_c}{N}&\sim 5\left(\frac{1+v^2}{v}\right)^{3/5} u^{3/5} (\sigma_2^2+\sigma_3^2)+u^{-2/5}\left(\frac{1+v^2}{v}\right)^{-2/5} \left(\frac{1+v^2}{v}\right) \left(\frac{du^2}{u}+u\sigma_1^2\right)\\
&=\left(\frac{1+v^2}{v}\right)^{3/5} \left(u^{3/5}(5(\sigma_2^2+\sigma_3^2)+\sigma_1^2)+u^{-7/5} du^2 \right)\\
&=dr^2 +\frac{9}{25}r^2\frac{(\sigma_1^2+5(\sigma_2^2+\sigma_3^2))}{4},
\end{align*}
and, for $\alpha\sim\pi/2$: 
\begin{align*}
\restr{g_c}{N}&\sim \left(\frac{1+v^2}{v}\right)^{3/5}\left(u^{3/5} (\sigma_1^2+\sigma_2^2+\sigma_3^2)+u^{-7/5}du^2\right)\\
&=dr^2+\frac{9}{25}r^2\frac{\left(\sigma_1^2+\sigma_2^2+\sigma_3^2\right)}{4},
\end{align*}
where, in both cases,
\[
r:=\frac{10}{3}\left(\frac{1+v^2}{v}\right)^{3/10}u^{3/10}.
\]
When $\alpha\sim\pi/2$, the link $S^3$ is endowed with the round metric, while, when $\alpha\sim0$, the round sphere is squashed by a factor $1/5$. 
\begin{remark}
Observe that $1/5$ is also the squashing factor on the round metric of $S^7$ that makes the space homogeneous, non-round and Einstein. It is well-known that there are no other metrics satisfying these properties \cite{Zil82}. 
\end{remark}

\section{The Cayley fibration invariant under the lift of the \texorpdfstring{$\Sp(1)\times\Id_1$}{Lg} action on \texorpdfstring{$S^4$}{Lg}} \label{Section Sp(1)xId_1}

Let $M:=\dS_{\!-}(S^4)$ and $M_0:=\R^{+}\times S^7$ be endowed with the torsion-free $\Spin(7)$-structures $\Phi_c$ constructed by Bryant and Salamon that we described in Section \ref{Section BS}. On each $\Spin(7)$ manifold, we construct the Cayley Fibration which is invariant under the lift to $M$ (or $M_0$) of the standard (left multiplication) $\Sp(1)\times\Id_1$ action on $S^4\subset \H\oplus\R$.

\begin{remark}
The exact same computations will work for the $\Sp(1)\times\Id_1$ action given by right multiplication of the quaternionic conjugate. In this case, the role of the north and of the south pole will be interchanged.
\end{remark}

\subsection{The choice of coframe on \texorpdfstring{$S^4$}{Lg}}\label{Subsection Frame S^4 Sp(1)xId_1}  As in Section \ref{Section SO(3)xId_2}, we choose an adapted orthonormal coframe on $S^4$ which is compatible with the symmetries we will impose. 

Consider $\R^5$ as the sum of a $4$-dimensional space $P\cong \H$ and its orthogonal complement $P^{\perp}\cong \R$. With respect to this splitting, we can write the $4$-dimensional unit sphere in the  following fashion: 
\begin{align*}
S^4=\left\{(\textbf{x},y)\in P\oplus P^{\perp} : \av{\textbf{x}}^2+\av{{y}}^2=1\right\}.
\end{align*}

Now, for all $(\textbf{x},y)\in S^4$ there exists a unique $\alpha\in[-\pi/2,\pi/2]$ such that 

\begin{align*}
\textbf{x}=\cos \alpha \textbf{u}, \hspace{20pt} {y}=\sin \alpha,
\end{align*}
for some $\textbf{u}\in S^3$. Note that $\textbf{u}$ is uniquely determined when $\alpha\neq\pm \pi/2$. Essentially, we are writing $S^4$ as a $1$-parameter family of $S^3$s that are collapsing to a point on each end of the parametrization. 

Let $\{\partial_1, \partial_2, \partial_3\}$ be the standard left-invariant orthonormal frame on $S^3\cong\Sp(1)$. Considering this frame in the description of $S^4$ above, we deduce that
\begin{align*}
f_0:=\partial_{\alpha}, \hspace{15pt}f_1:=\frac{\partial_{1}}{\cos\alpha},\hspace{15pt}f_2:=\frac{\partial_{2}}{\cos\alpha},\hspace{15pt}f_3:=\frac{\partial_{3}}{\cos\alpha},
\end{align*}
is an oriented orthonormal frame of $S^4\setminus \{\alpha=\pm\pi/2\}$. The dual coframe is: 
\begin{align}\label{frame S^4 Sp(1)xId_1}
b_0:=d\alpha;\hspace{15pt} b_1:=\cos\alpha \sigma_1;\hspace{15pt} b_2:=\cos\alpha \sigma_2; \hspace{15pt} b_3:=\cos\alpha\sigma_3,
\end{align}
where $\{\sigma_i\}_{i=1}^3$ is the dual coframe of $\{\partial_i\}_{i=1}^3$ in $S^3$, which is well-known to satisfy: 

\begin{align}\label{Maurer equations Sp(1)}
d 
\begin{pmatrix}
\sigma_1 \\ \sigma_2 \\ \sigma_3
\end{pmatrix}
=2\begin{pmatrix}
\sigma_2\wedge\sigma_3 \\ \sigma_3\wedge\sigma_1 \\ \sigma_1\wedge\sigma_2
\end{pmatrix}. 
\end{align}

We deduce that the round metric on the unit sphere $S^4$ can be written as: 
\begin{align*}
g_{S^4}=d\alpha^2+\cos^2\alpha g_{S^3},
\end{align*}
and the volume form is:
\begin{align*}
\vol_{S^4}=\cos^3\alpha d\alpha\wedge\vol_{S^3},
\end{align*}
where $g_{S^3}=\sigma_1^2+\sigma_2^2+\sigma_3^2$ and $\vol_{S^3}=\sigma_1\wedge\sigma_2\wedge\sigma_3$. 

\subsection{The horizontal and the vertical space}\label{The horizontal and the vertical space Sp(1)xId_1}
Exactly as in Subsection \ref{Subsection horizontal and vertical space SO(3)xId_2} we can compute the connection $1$-forms $\rho_i$ for $i=1,2,3$ with respect to the coframe we have constructed. Indeed, a straightforward computation involving (\ref{comprho}), (\ref{frame S^4 Sp(1)xId_1}) and (\ref{Maurer equations Sp(1)}) implies that $\rho_i=l \sigma_i$ for all $i=1,2,3$, where 
\begin{align*}
l:=\frac{\sin\alpha-1}{2}.
\end{align*}

Hence, we can deduce from (\ref{compVert1Forms}) that the vertical 1-forms in these coordinates are: 
\begin{equation}\label{Vertical1FormsinFrame Sp(1)xId_1}
\begin{aligned}
\xi_0&=da_0+l(a_1\sigma_1 + a_2\sigma_2+a_3\sigma_3 ), \hspace{20pt} &&\xi_1=da_1+l(-a_0\sigma_1 - a_2\sigma_3+a_3\sigma_2),\\
\xi_2&=da_2+l(-a_0\sigma_2+ a_1\sigma_3-a_3\sigma_1),  &&\xi_3=da_3+l(-a_0\sigma_3- a_1\sigma_2+a_2\sigma_1).
\end{aligned}
\end{equation}

\subsection{The \texorpdfstring{$\SU(2)$}{Lg} action} \label{Subsection Spin(3) action Sp(1)xId_1}

Given the splitting of $\R^5$ into $P\cong\H$ and its orthogonal complement $P^{\perp}$, we can consider $\SU(2)\cong\Sp(1)$ acting via left multiplication on $P$ and trivially on $P^{\perp}$. Equivalently, we are considering $\Sp(1)\cong \Sp(P) \times \Id_{P^{\perp}}\subset \SO(5)$. Being a subgroup of $\SO(5)$, the action descends to the unit sphere $S^4$. 

We first consider $\alpha\neq-\pi/2$, where we trivialize $S^4\setminus\{\textup{south pole}\}$ using homogeneous quaternionic coordinates on $\H\P^1\cong S^4$. In this chart, diffeomorphic to $\H$, the action is given by standard left multiplication. 

We extend the action on $S^4$ to the tangent bundle of $S^4$ via the differential. In this trivialization, $\H\times\H$, the action is given by left-multiplication on both factors. Hence, if we pick the trivialization of $P_{\SO(4)}$ induced by $\{1,i,j,k\}$, the action of $p\in\Sp(1)$ maps the element $(x,\Id_{\SO(4)})\in \H\times\SO(4)$ to $(p\cdot x,\tilde{p})$, where 

$$\tilde{p}=\left[
\begin{array}{c c c c}
p_0 &-p_1 &-p_2 &-p_3 \\ 

 p_1&  p_0 &-p_3 & p_2 \\ 
 p_2& p_3 & p_0 & -p_1 \\
 p_3 & -p_2 & p_1 & p_0
\end{array}
\right].$$

By the simply-connectedness of $\Sp(1)\cong\Spin(3)$, we  can lift the action to the spin structure $P_{\Spin(4)}$ of $S^4$. Using a similar diagram to (\ref{CDgloballifts SO(3)xId_2}) and the fact that the lift of $\tilde{p}$ is $(p,\Id_{\Sp(1)})\in \Sp(1)\times\Sp(1)$, we can show that in the trivialization of $P_{\Spin (4)}$, $\H\times\Sp(1)\times\Sp(1)$, the element $(x,(\Id_{\Sp(1)},\Id_{\Sp(1)}))$ is mapped to $(p\cdot x,(p,\Id_{\Sp(1)}))$.

 As in Section \ref{Section SO(3)xId_2}, this passes to the quotient space: $\dS_{\!-}(S^4)$, and, in the induced trivialization, $\H\times \H$, the action of $\Sp(1)$ is only given by left multiplication on the first factor by definition of $\mu_-$.

A similar argument works for the other chart of $\H\P^1$. However, the left multiplication becomes right multiplication of the conjugate, and the lift of the new $\tilde{p}$ is $(\Id_{\Sp(1)},p)$. It follows that $\Sp(1)$ acts on the fibre over the south pole as it acts on $\H$.

 In particular, we proved the following lemma.

\begin{lemma} \label{Geometry of Spin(3) orbits Sp(1)xId_1}
The orbits of the $\SU(2)$ action on $\dS_{\!-}(S^4)$ are given in Table \ref{Table Orbits Sp(1)xId_1}.
\end{lemma}
\begin{table}[h]
\[
\begin{array}{|c|c|c|c|}
\hline
\alpha & a & \textup{Orbit}\\ \hline
 \neq\pm\frac{\pi}{2} &  & S^3 \\ \hline
 =-\frac{\pi}{2} & \neq0 & S^3 \\ \hline
 =-\frac{\pi}{2} & =0 & \textup{Point} \\ \hline
 =\frac{\pi}{2} & & \textup{Point} \\ \hline
 
\end{array}
\] 
\caption{$\Spin(3)$ Orbits}\label{Table Orbits Sp(1)xId_1}
\end{table}

When $\alpha\neq\pm\pi/2$ we can use the orthonormal frame of Subsection \ref{Subsection Frame S^4 Sp(1)xId_1}. Obviously, it is invariant under the action. Hence, in the induced trivialization of $\dS_{\!-}(S^4)$, $\Sp(1)$ acts only on the component of the basis. In particular, it follows that $\{\sigma_1,\sigma_2,\sigma_3\}$ is a coframe on the orbits of the $\SU(2)$ action, and, $\{\partial_1,\partial_2,\partial_3\}$ is the relative frame. Observe that we are working on the coframe $\{d\alpha, \sigma_1,\sigma_2,\sigma_3, da_0, da_1, da_2, da_3\}$.

\subsection{The choice of frame and the \texorpdfstring{$\Spin(7)$}{Lg} geometry in the adapted coordinates}\label{Subsection final coframe Sp(1)xId_1} Since the considered $\SU(2)$ action only moves the base of the vector bundle $\dS_{\!-}(S^4)$ in the trivialization of Subsection \ref{Subsection Frame S^4 Sp(1)xId_1}, it is natural to use: $\{d\alpha,\sigma_1,\sigma_2,\sigma_3,\xi_0,\xi_1,\xi_2,\xi_3\}$. The metrics $g_c$ and the Cayley forms $\Phi_c$ admit a nice formula with respect to this coframe. Recall that we are working on the chart $\mathcal{U}:=\dS_{\!-}(S^4)\setminus \{\alpha=\pm\pi/2\}$.

\begin{proposition}
Given $c\geq0$, the Riemannian metric $g_c$, in the coframe considered in this subsection, satisfies: 
\begin{align}\label{g_c final coordinates Sp(1)xId_1}
g_c=5(c+r^2)^{3/5}\left( d\alpha^2+\cos^2\alpha \left(\sigma_1^2+\sigma_2^2+\sigma_3^2\right)\right)+4(c+r^2)^{-2/5} \left(\xi_0^2+\xi_1^2+\xi_2^2+\xi_3^2\right),
\end{align}
where $r^2=a_0^2+a_1^2+a_2^2+a_3^2$.

Given $c\geq0$, the Cayley form $\Phi_c$, in the coframe considered in this subsection, satisfies: 
\begin{equation}\label{Phi_c final coordinates Sp(1)xId_1}
\begin{aligned}
\Phi_c=&16(c+r^2)^{-4/5}\xi_0\wedge\xi_1 \wedge\xi_2\wedge\xi_3+25 (c+r^2)^{6/5}\cos^3\alpha d\alpha\wedge\sigma_1\wedge\sigma_2\wedge\sigma_3\\
&+20(c+ r^2)^{1/5} \cos\alpha\left(\sum_{i=1}^3 (\xi_0\wedge\xi_i-\xi_j\wedge\xi_k)\wedge (d\alpha\wedge\sigma_i-\cos\alpha\sigma_j\wedge\sigma_k)\right),
\end{aligned}
\end{equation}
where $r^2=a_0^2+a_1^2+a_2^2+a_3^2$.
\end{proposition}
\begin{proof}
It follows immediately from (\ref{Phi_c}), (\ref{g_c}) and the choice of the coframe.
\end{proof}
If we denote by $\{e_\alpha,e_1,e_2,e_3,e_{\xi_0},e_{\xi_1},e_{\xi_2},e_{\xi_3}\}$ the frame dual to $\{d\alpha,\sigma_1,\sigma_2,\sigma_3,\xi_0,\xi_1,\xi_2,\xi_3\}$, it is straightforward to relate these vectors to $\partial_{\alpha}, \partial_1,\partial_2,\partial_3,\partial_{a_0},\partial_{a_1},\partial_{a_2},\partial_{a_3}$. 
\begin{lemma} \label{e_i w.r.t partial_i Sp(1)xId_1}
The dual frame $\{e_\alpha,e_1,e_2,e_3,e_{\xi_0},e_{\xi_1},e_{\xi_2},e_{\xi_3}\}$ satisfies:
\begin{equation*} 
\begin{split}
e_\alpha&=\partial_\alpha;\\
e_2&=\partial_2 +l\left( -a_2\partial_{a_0} -a_3\partial_{a_1}+a_0\partial_{a_2}+a_1\partial_{a_3}\right); \\
e_{\xi_i}&=\partial_{a_i} \hspace{15 pt} \forall i=0,1,2,3, \\
\end{split}
\quad\quad
\begin{split}
e_1&=\partial_1 +l\left( -a_1\partial_{a_0} +a_0\partial_{a_1}+a_3\partial_{a_2}-a_2\partial_{a_3}\right);\\
e_3&=\partial_3 +l\left( -a_3\partial_{a_0} +a_2 \partial_{a_1}-a_1\partial_{a_2}+a_0\partial_{a_3}\right);\\
\hspace{5pt}\\
\end{split}
\end{equation*}
where $l$ is as defined in Subsection \ref{The horizontal and the vertical space Sp(1)xId_1}.
\end{lemma}
\begin{proof}
It is straightforward from the definition of dual frame and (\ref{Vertical1FormsinFrame Sp(1)xId_1}).
\end{proof}

\subsection{The Cayley condition}\label{subsection Cayley condition Sp(1)xId_1} Analogously to the case carried out in Section \ref{Section SO(3)xId_2}, the generic orbits of the considered $\SU(2)$ action are $3$-dimensional (see Lemma \ref{Geometry of Spin(3) orbits Sp(1)xId_1}). Hence, it is sensible to look for invariant Cayley submanifolds. To this purpose, we assume that the submanifold $N$ consists of a 1-parameter family of $3$-dimensional $\SU(2)$-orbits in $M$. In particular, the coordinates that do not describe the orbits, i.e. $a_0,a_1,a_2, a_3$ and $\alpha$, need to be functions of a parameter $\tau$. This means that we can write:
\begin{equation}
\begin{aligned}
N=&\left\{\left( (\cos\alpha(\tau)\textbf{u},\sin\alpha(\tau)),( a_0(\tau),a_1(\tau),a_2(\tau),a_3(\tau)) \right) : \av{u}=1,\tau\in(-\epsilon,\epsilon) \right\}.
\end{aligned}
\end{equation}
The tangent space is spanned by $\{\partial_1,\partial_2,\partial_3,\dot{\alpha}\partial_\alpha+\sum_{i=0}^3 \dot{a}_i \partial_{a_i}\}$, where the dots denote the derivatives with respect to $\tau$. The condition under which $N$ is Cayley becomes a system of ODEs.

\begin{theorem}\label{Complicated ODEs Sp(1)xId_1}
Let $N$ be an $\SU(2)$-invariant submanifold as described at the beginning of this subsection. Then, $N$ is Cayley in the chart $\mathcal{U}$ if and only if the following system of ODEs is satisfied: 
\begin{equation*}
\left\{
\begin{aligned}
&\dot{a}_0 a_1-\dot{a}_1 a_0-\dot{a}_2 a_3+\dot{a}_3 a_2=0\\
&\dot{a}_0 a_2+\dot{a}_1 a_3-\dot{a}_2 a_0-\dot{a}_3 a_1=0\\
&\dot{a}_0 a_3-\dot{a}_1 a_2+\dot{a}_2 a_1-\dot{a}_3 a_0=0\\
&\cos\alpha (-f\cos^2\alpha+3l^2 g r^2)\dot{a}_0-l(l^2gr^2-3f\cos^2 \alpha)a_0\dot{\alpha}=0\\
&\cos\alpha (-f\cos^2\alpha+3l^2 g r^2)\dot{a}_1-l(l^2gr^2-3f\cos^2 \alpha)a_1\dot{\alpha}=0\\
&\cos\alpha (-f\cos^2\alpha+3l^2 g r^2)\dot{a}_2-l(l^2gr^2-3f\cos^2 \alpha)a_2\dot{\alpha}=0\\
&\cos\alpha (-f\cos^2\alpha+3l^2 g r^2)\dot{a}_3-l(l^2gr^2-3f\cos^2 \alpha)a_3\dot{\alpha}=0
\end{aligned}
\right.
\end{equation*}
where $r^2=a_0^2+a_1^2+a_2^2+a_3^2$, $l=(\sin\alpha-1)/2$, $f=5(c+r^2)^{3/5}$ and $g=4(c+r^2)^{-2/5}$.
\end{theorem}
\begin{proof}
We first write the tangent space of $N$, which is spanned by $\{\partial_1,\partial_2,\partial_3,\dot{\alpha}\partial_\alpha+\sum_{i=0}^3 \dot{a}_i \partial_{a_i}\}$, in terms of the frame $\{e_\alpha,e_1,e_2,e_3,e_{\xi_0},e_{\xi_1},e_{\xi_2},e_{\xi_3}\}$. This can be easily done using Lemma \ref{e_i w.r.t partial_i Sp(1)xId_1}. Through a long computation analogous to the one carried out in Appendix \ref{appendix proof of Theorem ODES SO(3)xId_2}, we can apply Proposition \ref{Cayley condition KM} to this case, and we obtain the system of ODEs.
\end{proof}

\begin{remark}
It is interesting to point out that, exactly as in the $\SO(3)\times\Id_2$ case (see Lemma \ref{Basis Lambda^2_7 SO(3)xId_2}), the projection $\pi_7$ of Proposition \ref{Cayley condition KM} will just be the identity in the proof of Theorem \ref{Complicated ODEs Sp(1)xId_1}.
\end{remark}

\subsection{The Cayley fibration}
 In the previous section we found the condition that makes $N$, $\SU(2)$-invariant submanifold, Cayley. This consists of a system of ODEs, which will characterize the desired Cayley fibration. 
 
Harvey and Lawson local existence and uniqueness theorem implies that any $\SU(2)$-invariant Cayley can meet the zero section only when $\alpha=\pm\pi/2$, i.e. outside of $\mathcal{U}$. Otherwise, the zero section of $\dS_{\!-}(S^4)$, which is Cayley, would intersect such an $N$ in a $3$-dimensional submanifold, contradicting Harvey and Lawson theorem. It follows that the initial value of one of the $a_i$s is different from zero. We take $a_0(0)\neq0$, as the other cases will follow similarly. 
Now, it is straightforward to notice that:
\begin{align}\label{solution first ODEs Sp(1)xId_1}
a_1=\frac{a_1(0)}{a_0(0)} a_0; \hspace{15pt} a_2=\frac{a_2(0)}{a_0(0)} a_0; \hspace{15pt} a_3=\frac{a_3(0)}{a_0(0)} a_0;
\end{align}
solves the first $3$ equations of the system given in Theorem \ref{Complicated ODEs Sp(1)xId_1}. Moreover, it also reduces the remaining equations to the ODE: 
\begin{align*}
\cos\alpha (-f\cos^2\alpha+3l^2 g r^2)\dot{a}_0-l(l^2gr^2-3f\cos^2 \alpha)a_0\dot{\alpha}=0,
\end{align*}
where, as usual, $r^2=a_0^2+a_1^2+a_2^2+a_3^2$, $l=(\sin\alpha-1)/2$, $f=5(c+r^2)^{3/5}$ and $g=4(c+r^2)^{-2/5}$. As (\ref{solution first ODEs Sp(1)xId_1}) implies that $a_0=p^{-1} r$, where $p$ is the positive real number satisfying $p^2=1+\sum_{i=1}^3(a_i(0)/a_0(0))^2$, we can rewrite the previous ODE as:
\begin{align}\label{final ODE Sp(1)xId_1}
\cos\alpha (-f\cos^2\alpha+3l^2 g r^2)\dot{r}-l(l^2gr^2-3f\cos^2 \alpha)r\dot{\alpha}=0.
\end{align}

\begin{remark}
It is easy to verify that (\ref{final ODE Sp(1)xId_1}) is not in exact form. Hence, it cannot be easily integrated. It is a non-trivial open task to verify whether, possibly up to change of coordinates, (\ref{final ODE Sp(1)xId_1}) can be integrated in closed form. 
 \end{remark}

In order to understand the $\SU(2)$-invariant Cayley fibrations, we analyse the ODE (\ref{final ODE Sp(1)xId_1}). First, we deduce the sign of $f_1:=\cos\alpha (-f\cos^2\alpha+3l^2 g r^2)$. If we let 
\begin{align*}
\alpha_c(r):=\arcsin\left(-\frac{2r^2+5c}{8r^2+5c}\right),
\end{align*}
it easy to verify that $f_1$ is positive on the left of $\alpha_c$ for $(\alpha,r)\in(-\pi/2,\pi/2)\times \R^{+}$, and negative otherwise. Moreover, $f_1$ vanishes along the $3$ curves $\alpha_c,\alpha=\pm\pi/2$; there, $f_1$ changes sign. Note that $\alpha_c\to \arcsin(-1/4)$ as $r\to\infty$. 

Now, we consider $f_2:=l(l^2gr^2-3f\cos^2 \alpha)r$. Letting
$$
\beta_c(r):= \arcsin\left(-\frac{14r^2+15c}{16r^2+15c}\right),
$$
then, $f_2$ is positive on the right of $\beta_c$ for $(\alpha,r)\in(-\pi/2,\pi/2)\times \R^{+}$, and it is negative otherwise. Obviously, $f_2$ vanishes along the curve $\beta_c$ and the vertical line $\alpha=\pi/2$. Note that $\beta_c\to \arcsin(7/8)$ as $r\to\infty$. The last key observation is that $f_2/f_1$ tends to zero as $\alpha$ tends to $\pi/2$.

Putting what said so far together, and observing that $\beta_c(r)<\alpha_c(r)$ for all $r>0$, we can draw the flow lines for $(\ref{final ODE Sp(1)xId_1})$ (see Figure \ref{Flow lines vector fields Sp1xId1}). 
\begin{figure}
    \centering
      \includegraphics[scale=0.4]{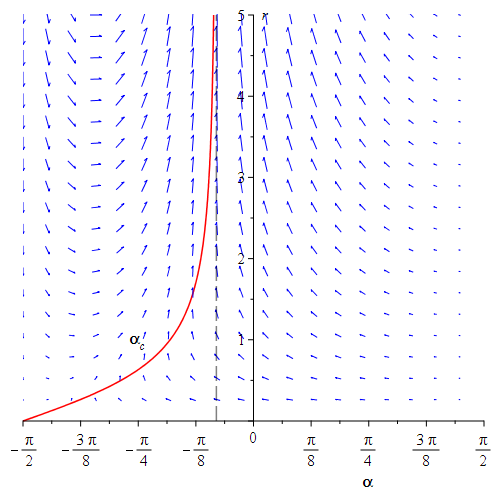}
    \caption{Flow lines for $(\ref{final ODE Sp(1)xId_1})$.}%
    \label{Flow lines vector fields Sp1xId1}%
\end{figure}
Finally, we can use these to deduce the form of the solutions from standard arguments (see Figure \ref{Figure solution ODEs Sp1xId1}). We give further details in Appendix \ref{Appendix analysis ODE Sp(1)xId_1}.

\begin{figure}
    \centering
    \begin{subfigure}{0.48\textwidth}\centering\includegraphics[scale=0.42]{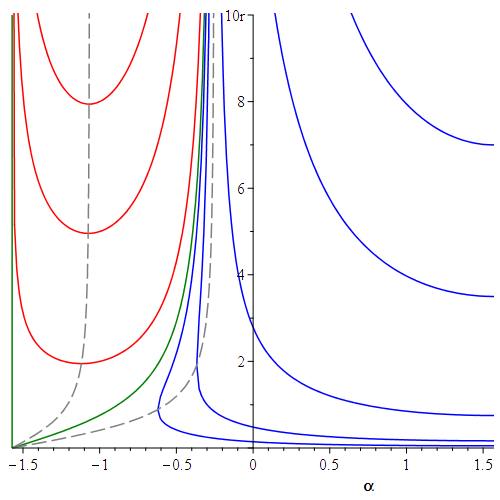} \caption{generic case}
    \end{subfigure}
	\begin{subfigure}{0.48\textwidth}\centering\includegraphics[scale=0.42]{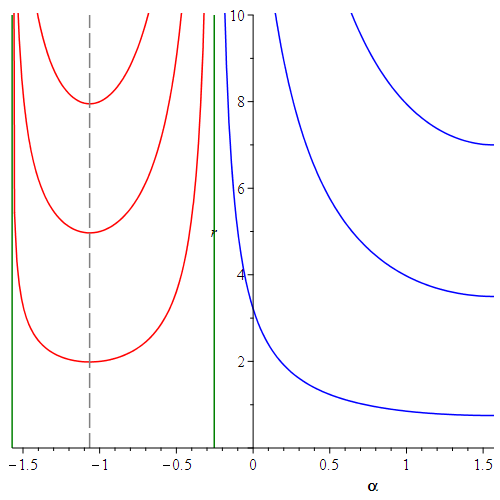} \caption{conical case}
    \end{subfigure}
    \caption{Solutions of $(\ref{final ODE Sp(1)xId_1})$.}%
    \label{Figure solution ODEs Sp1xId1}%
\end{figure}

\subsection*{The conical version.} We consider the easier conical case first. From a topological point of view, it is obvious that the red and green Cayleys of Figure \ref{Figure solution ODEs Sp1xId1} (B) are homeomorphic to $S^3\times\R$. As the the group action becomes trivial on $\alpha=\pi/2$, the topology of the fibres in blue cannot be recovered from the picture. However, it will be clear from the asymptotic analysis that these are smooth topological $\R^4$s. As a consequence, we have constructed a Cayley fibration on the chart $\mathcal{U}\cap M_0$, which extends to the whole $M_0$ by continuity (i.e. we complete the Cayleys in blue and we add the whole $\pi_0$-fibre at $\alpha=-\pi/2$). On $M_0$ the Cayley fibration remains a fibration in the classical sense. A reasoning similar to the one of Section \ref{Section SO(3)xId_2} shows that the parametrizing space $\mathcal{B}$ of the Cayley fibration is $\R^4$. 

\subsection*{The smooth version.} Now, we deal with the generic case $c>0$. As above, the topology of the red Cayleys of Figure \ref{Figure solution ODEs Sp1xId1} (A) is $S^3\times \R$; the blue ones have topology $\R^4$. In the latter, we use the same asymptotic analysis argument of the conical case. Finally, the submanifolds in green are smooth topological $\R^4$s. As usual, we extend the Cayley fibration on $\mathcal{U}$ to the whole $M$ by continuity (i.e. we add the whole $\pi_c$-fibre over $\alpha=-\pi/2$, we complete the Cayleys in blue and green, and we add the zero section $S^4$). Observe that the zero section, the $\pi_c$-fibre over $\alpha=-\pi/2$ and the green Cayleys all intersect in a point $p$. It follows that the $M'$ given in Definition \ref{Cayley fibration definition} is equal to $M\setminus\{p\}$.  Once again, a reasoning similar to the one of Section \ref{Section SO(3)xId_2} shows that the parametrizing space $\mathcal{B}$ of the Cayley fibration is $S^4$. 

\subsection*{The smoothness of the fibres (the asymptotic analysis)} In this subsection, we study the smoothness of the fibres. This is trivial as long as the submanifolds are contained in $\mathcal{U}$; hence, the Cayleys of topology $S^3\times \R$ are smooth, and we only need to check the others at the points where they meet $\partial\mathcal{U}$. To this purpose, we carry out a asymptotic analysis similar to the one of Section \ref{Section SO(3)xId_2}. 

As a first step, we restrict the metric $g_c$ to $N$. Combining $(\ref{g_c final coordinates Sp(1)xId_1})$ together with $(\ref{solution first ODEs Sp(1)xId_1})$ and its consequence $a_0=p^{-1} r$ for $p$ positive real number satisfying $p^2=1+\sum_{i=1}^3(a_i(0)/a_0(0))^2$, we can write the restriction as follows: 
\begin{equation} \label{metric for asymptotic/singular geometry Sp(1)xId_1}
\begin{aligned}
\restr{g_c}{N}\!=\!\left( 5(c+r^2)^{3/5} \cos^2\alpha +4(c+r^2)^{-2/5} l^2 r^2 \right)\!\left(\sigma_1^2+\sigma_2^2+\sigma_3^2\right)\!+\!4(c+r^2)^{-2/5}dr^2\!+\!5(c+r^2)^{3/5} d\alpha^2,
\end{aligned}
\end{equation}
where $\alpha$ and $r$ are related by the differential equation (\ref{final ODE Sp(1)xId_1}) and, as usual, $l=(\sin\alpha-1)/2$.

Recall that $f_2/f_1\to 0$ as $\alpha\to\pi/2$. Therefore, the Cayleys around $\alpha=\pi/2$ are asymptotic to the horizontal line $\alpha=r_0$ for some constant $r_0\geq0$.  By  (\ref{metric for asymptotic/singular geometry Sp(1)xId_1}), the metric in this first order linear approximation becomes: 
\begin{equation*}
\begin{aligned}
\restr{g_c}{N}\sim 5(c+r_0^2)^{3/5} \left( d(\alpha-\pi/2)^2+(\alpha-\pi/2)^2 \left(\sigma_1^2+\sigma_2^2+\sigma_3^2\right)\right).
\end{aligned}
\end{equation*}
In this way, we have proved that near $\alpha=\pi/2$ every Cayley we have constructed is smooth. Moreover, we can also deduce that the blue Cayleys of Figure \ref{Figure solution ODEs Sp1xId1} are topologically $\R^4$s.

Finally, we need to check whether the remaining Cayleys of topology $\R^4$ are smooth or not. In this situation we can approximate them near $\alpha=-\pi/2$ with the submanifold associated to the line: 
$$
\alpha=Ar-\frac{\pi}{2},
$$
where $A$ is some positive constant (as the lines corresponding to the Cayleys live between $\alpha_c$ and $\beta_c$). The metric in the linear approximation is asymptotic to: 
\begin{equation*}
\begin{aligned}
\restr{g_c}{N} \sim c^{-2/5}(5cA^2+4) \left( dr^2+r^2 \left(\sigma_1^2+\sigma_2^2+\sigma_3^2\right)\right),
\end{aligned}
\end{equation*}
hence, we conclude that these submanifolds are smooth as well.

\subsection*{The main theorems} Putting all these results together we obtain the following theorems. 
\begin{theorem}[Generic case]\label{Main theorem Cayley fibration generic Sp(1)xId_1}
Let $(M,\Phi_c)$ be the Bryant--Salamon manifold constructed over the round sphere $S^4$ for some $c>0$, and let $\SU(2)$ act on $M$ as in Subsection \ref{Subsection Spin(3) action Sp(1)xId_1}. Then, $M$ admits an $\SU(2)$-invariant Cayley fibration parametrized by $\mathcal{B}\cong S^4$. The fibres are topologically $S^3\times \R$, $S^4$ and $\R^4$. All the Cayleys are smooth. There is only one point where multiple fibres intersect. This point lies in the zero section of $\dS_{\!-}(S^4)$, and there are $S^3\sqcup\{\textup{two points}\}$ Cayleys passing through it.
\end{theorem}
\begin{theorem}[Conical case]\label{Main theorem Cayley fibration conical Sp(1)xId_1}
Let $(M_0,\Phi_0)$ be the conical Bryant--Salamon manifold constructed over the round sphere $S^4$, and let $\SU(2)$ act on $M_0$ as in Subsection \ref{Subsection Spin(3) action Sp(1)xId_1}. Then, $M_0$ admits an $\SU(2)$-invariant Cayley fibration parametrized by $\mathcal{B}\cong \R^4$. The fibres are topologically $S^3\times\R$ or $\R^4$ and are all smooth. Moreover, as these do not intersect, the $\SU(2)$-invariant Cayley fibration is a fibration in the usual differential geometric sense with fibres Cayley submanifolds. 
\end{theorem}

\begin{remark}
	Blowing-up at the north pole, it is easy to see that the Cayley fibration becomes trivial in the limit. 
\end{remark}

\begin{remark}
	As in the previous section, we are able to compute the multi-moment maps relative to this action explicitly. Indeed, this is:
	\begin{align*}
\nu_c:=\frac{5}{6}(r^2-5c)(c+r^2)^{1/5} (\sin\alpha-1)^3-\frac{25}{2}(c+r^2)^{6/5}\cos^2\alpha (\sin\alpha-1).
\end{align*}
\end{remark}

In order to provide an idea on how the multi-moment maps behave, we draw the level sets of $\nu_1$ and $\nu_0$ (see Figure \ref{multi-moment map Sp(1)xId_1}).

\begin{figure}[H] 
    \centering
    \begin{subfigure}{0.48\textwidth}\centering\includegraphics[scale=0.43]{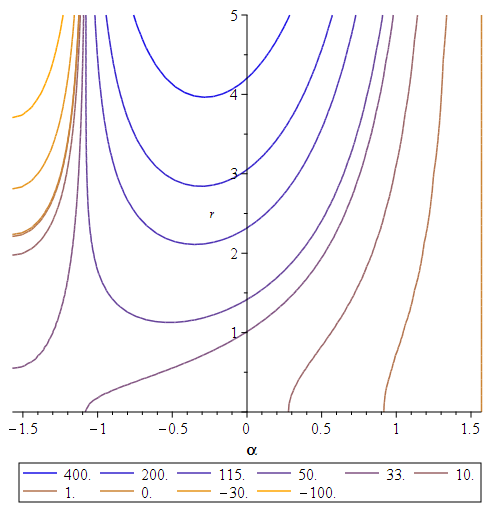} \caption{Level sets of $\nu_1$}
    \end{subfigure}
    \begin{subfigure}{0.48\textwidth}\centering\includegraphics[scale=0.43]{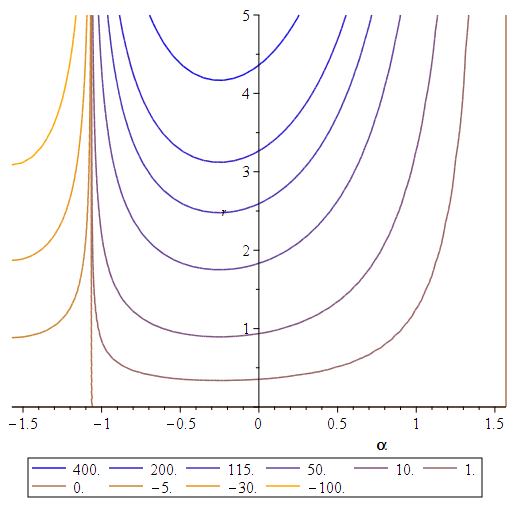} \caption{Level sets of $\nu_0$}
    \end{subfigure}
    \caption{Level sets of the multi-moment map in the generic and conical case}%
    \label{multi-moment map Sp(1)xId_1}%
\end{figure}

\subsection*{Asymptotic geometry.} The first observation we need to make is that there are only two asymptotic behaviours for the Cayleys constructed in Theorem \ref{Main theorem Cayley fibration generic Sp(1)xId_1} and in Theorem \ref{Main theorem Cayley fibration conical Sp(1)xId_1}: one corresponding to $\alpha\sim -\pi/2$ and the other to $\alpha\sim\arcsin(-1/4)$. In both cases, we can use (\ref{metric for asymptotic/singular geometry Sp(1)xId_1}) to obtain the asymptotic cone, which is: 
\begin{align*}
\restr{g_c}{N}\sim ds^2+\frac{9}{25}s^2 (\sigma_1^2+\sigma_2^2+\sigma_3^2),
\end{align*}
for $\alpha\sim\pi/2$, and it is 
\begin{align*}
\restr{g_c}{N}\sim ds^2+\frac{9}{16}s^2 (\sigma_1^2+\sigma_2^2+\sigma_3^2),
\end{align*}
for $\alpha\sim\arcsin(-1/4)$, where $s:=(10/3)r^{3/5}$.

\appendix

\section{}\label{appendix proof of Theorem ODES SO(3)xId_2}

In this appendix, we prove Theorem \ref{Complicated ODEs SO(3)xId_2}. First, we need to rewrite the tangent space of $N$ in the diagonalizing frame of Subsection \ref{Subsection final coframe SO(3)xId_2}.
\begin{lemma}\label{uvwy def SO(3)xId_2}
The tangent space of $N$ is spanned by:
$$
u:=te_{\omega_2}-se_{\omega_1}, \hspace{10pt} v:=e_2+\frac{\sin{\alpha}}{2}(te_s-se_t), \hspace{10pt} w:=e_3+\sin\alpha\left(te_{\omega_1}+se_{\omega_2}\right)
$$
and
$$y:=\dot{s}e_s+\dot{t}e_t+\dot{\alpha}e_\alpha+\dot{\beta}e_\beta+\dot{\delta}\left(se_{\omega_1}+te_{\omega_2}\right).$$
Moreover, through the musical isomorphism, we have:
\begin{align*}
u^{\musFlat}=(c+r^2)^{-2/5}(t\omega_2-s{\omega_1}), \hspace{10pt} v^{\musFlat}=5(c+r^2)^{3/5}\cos^2\alpha\sigma_2+2(c+r^2)^{-2/5}\sin{\alpha}(t\tilde{ds}-s\tilde{dt}),
\end{align*}
\begin{align*}
w^{\musFlat}=5(c+r^2)^{3/5} \cos^2\alpha \sigma_3+(c+r^2)^{-2/5}\sin\alpha\left(t{\omega_1}+s{\omega_2}\right)
\end{align*}
and 
\begin{align*}
y^{\musFlat}=5(c+r^2)^{3/5}(\dot{\alpha}d\alpha+\sin^2\alpha\dot{\beta}d\beta)+4(c+r^2)^{-2/5}(\dot{s}\tilde{ds}+\dot{t}\tilde{dt})+(c+r^2)^{-2/5}\dot{\delta}(s\omega_1+t\omega_2),
\end{align*}
where $r^2=s^2+t^2$.
\end{lemma}
\begin{proof}
One can immediately see from Lemma \ref{e_i w.r.t partial_i SO(3)xId_2} that $\partial_1=u$, $\partial_2=v$ and $\partial_\delta=se_{\omega_1}+te_{\omega_2}$. We use these equality to obtain: 
\begin{align*}
(s^2+t^2)\partial_\delta-(t^2-s^2)\partial_1&=(s^2+t^2)(se_{\omega_1}+te_{\omega_2})-(t^2-s^2)(te_{\omega_2}-se_{\omega_1})\\
&=2st(te_{\omega_1}+se_{\omega_2}),
\end{align*}
which implies that $\partial_3=w$. We conclude noticing that $\dot{s}\partial_s+\dot{t}\partial_t+\dot{\alpha}\partial_\alpha+\dot{\beta}\partial_\beta+\dot{\delta}\partial_{\delta}=y-\dot{\beta}\cos\alpha\partial_1$, where we used once again Lemma \ref{e_i w.r.t partial_i SO(3)xId_2}. Obviously, the space spanned by $\{u,v,w,y\}$ coincides with the one spanned by $\{u,v,w,y-\dot{\beta}\cos\alpha\partial_1\}$.

The second part of the Lemma follows immediately from Proposition \ref{Proposition Diagonal metric SO(3)xId_2}, where we proved that the metric is diagonal in this frame. 
\end{proof}


Let $B$ be as in Proposition \ref{Cayley condition KM}. We compute the terms of $B$ in the basis $\{u,v,w,y\}$.
\begin{lemma} \label{B in uvwy SO(3)xId_2}
Let $u,v,w,y$ as in Lemma \ref{uvwy def SO(3)xId_2}. Then, we have:
\begin{align*}
B(v,w,y)&=
\begin{aligned}[t]
&25(c+r^2)^{6/5}\sin\alpha\cos^2\alpha (\dot{\beta}d\alpha-\dot{\alpha}d\beta)\\
&+2\sin^2\alpha(c+r^2)^{-4/5} \left( (t\dot{t}+s\dot{s})(t\omega_2-s\omega_1)-(t^2-s^2)\dot{\delta}(t\tilde{dt}+s\tilde{ds})\right)\\
&+5(c+r^2)^{1/5}\bigg( 2\cos^2\alpha \left(\dot{s}\omega_1-\dot{t}\omega_2+\dot{\delta}(t\tilde{dt}-s\tilde{ds})+\sin\alpha (ts\dot{\delta}\sigma_2+(s\dot{t}-t\dot{s})\sigma_3)\right)\\
&+2\sin\alpha\cos\alpha\left( (s^2-t^2)\dot{\delta}d\alpha+\dot{\alpha}(t\omega_2-s\omega_1)\right)+(s^2+t^2)\sin^3\alpha(\dot{\alpha}d\beta-\dot{\beta}d\alpha)\\
&+4\cos\alpha\sin^2\alpha \left(\dot{\beta}(s\tilde{ds}+t\tilde{dt})-(s\dot{s}+t\dot{t})d\beta\right)\bigg),
\end{aligned}\\
B(w,u,y)&=
\begin{aligned}[t]
&4(c+r^2)^{-4/5}(t^2+s^2)\sin\alpha(\dot{t}\tilde{ds}-\dot{s}\tilde{dt})\\
&+5(c+r^2)^{1/5}\big(-2\cos^2\alpha(s\dot{s}+t\dot{t})\sigma_2-2\cos\alpha\sin\alpha st\dot{\delta}d\beta+\cos\alpha\sin\alpha\dot{\beta}(t\omega_1+s\omega_2)\\
&+2\cos\alpha(s\dot{t}-t\dot{s})d\alpha+2\cos\alpha\dot{\alpha}(t\tilde{ds}-s\tilde{dt})+\cos\alpha\sin\alpha(t^2+s^2)\dot{\alpha}\sigma_2\\
&-\cos\alpha\sin^2\alpha(t^2+s^2)\dot{\beta}\sigma_3\big),
\end{aligned}\\
B(u,v,y)&=
\begin{aligned}[t]
&2(c+r^2)^{-4/5} \sin\alpha(-2\dot{\delta}st(t\tilde{dt}+s\tilde{ds})+(t\dot{t}+s\dot{s})(t\omega_1+s\omega_2))\\
&5(c+r^2)^{1/5}\big( -2\cos^2\alpha(s\dot{s}+t\dot{t})\sigma_3-2\cos\alpha st\dot{\delta}d\alpha+\cos\alpha\dot{\alpha}(s\omega_2+t\omega_1)\\
&+2\cos\alpha\sin\alpha(t\dot{s}-s\dot{t})d\beta+2\cos\alpha\sin\alpha\dot{\beta}(s\tilde{dt}-t\tilde{ds})+(s^2+t^2)\cos\alpha\sin\alpha\dot{\alpha}\sigma_3\\
&+(s^2+t^2)\cos\alpha\sin^2\alpha\dot{\beta}\sigma_2\big),
\end{aligned}\\
B(v,u,w)&=
\begin{aligned}[t]
&2(c+r^2)^{-4/5}\sin^2\alpha (t^2+s^2)(t\tilde{dt}+s\tilde{ds})+10(c+r^2)^{1/5}\big(-\cos^2\alpha(s\tilde{ds}+t\tilde{dt})\\
&+\sin\alpha\cos\alpha(t^2+s^2)d\alpha\big),
\end{aligned}
\end{align*}
where $B$ is defined in Proposition \ref{Cayley condition KM} and $r^2=s^2+t^2$.
\end{lemma}
\begin{proof}
The multilinearity of the Cayley form $\Phi_c$ implies that the same property holds for $B$. Now, expanding the formula (\ref{Phi_c final coordinates SO(3)xId_2}) for $\Phi_c$, we obtain: 
\begin{align*}
\Phi_c=&4(c+r^2)^{-4/5} \tilde{ds}\wedge\tilde{dt}\wedge \omega_2\wedge \omega_1+25(c+r^2)^{6/5}\sin\alpha\cos^2\alpha d\alpha\wedge d\beta\wedge\sigma_3\wedge\sigma_2\\
&10(c+r^2)^{1/5} \bigg(\sin\alpha \tilde{ds}\wedge\omega_1\wedge d\alpha\wedge d\beta+\cos^2\alpha \tilde{ds}\wedge \omega_1\wedge \sigma_2\wedge\sigma_3-\sin\alpha\tilde{dt}\wedge\omega_2\wedge d\alpha\wedge d\beta\\
&-\cos^2\alpha\tilde{dt}\wedge\omega_2\wedge\sigma_2\wedge\sigma_3+2\cos\alpha\tilde{ds}\wedge\tilde{dt}\wedge d\alpha\wedge \sigma_2-2\cos\alpha\sin\alpha\tilde{ds}\wedge\tilde{dt}\wedge d\beta\wedge \sigma_3\\
&+\frac{\cos\alpha}{2}\omega_1\wedge\omega_2\wedge d\alpha\wedge\sigma_2-\frac{\cos\alpha\sin\alpha}{2}\omega_1\wedge\omega_2\wedge d\beta\wedge \sigma_3-\cos\alpha \tilde{ds}\wedge \omega_2\wedge d\alpha\wedge \sigma_3\\
&-\cos\alpha\sin\alpha \tilde{ds}\wedge \omega_2\wedge d\beta \wedge \sigma_2-\cos\alpha\tilde{dt}\wedge \omega_1\wedge d\alpha \wedge \sigma_3-\cos\alpha\sin\alpha \tilde{dt}\wedge \omega_1\wedge d\beta \wedge \sigma_2\bigg).
\end{align*}
It is straightforward to conclude using the definition of $B$. 
\end{proof}

Consider the two-form given in Proposition \ref{Cayley condition KM} that projects to $\eta$ through $\pi_7$. The summands of such two form can be computed through a direct computation involving the terms obtained in Lemma \ref{uvwy def SO(3)xId_2} and Lemma \ref{B in uvwy SO(3)xId_2}.

\begin{corollary}\label{definition Psi_i in SO(3)xId_2}
Let $u,v,w,y$ as in Lemma \ref{uvwy def SO(3)xId_2} and let $\Psi_1:=u^{\musFlat}\wedge B(v,w,y)$, $\Psi_2=v^{\musFlat}\wedge B(w,u,y)$, $\Psi_3=w^{\musFlat}\wedge B(u,v,y)$, $\Psi_4=y^{\musFlat}\wedge B(v,u,w)$, where $B$ is as defined in Proposition \ref{Cayley condition KM}. Then, we have:
\begin{align*}
\Psi_1&=
\begin{aligned}[t] 
&25(c+r^2)^{4/5}\sin\alpha\cos^2\alpha (t\omega_2-s\omega_1)\wedge (\dot{\beta} d\alpha-\dot{\alpha}d\beta)\\
&-(c+r^2)^{-6/5} 2\sin^2 \alpha (t^2-s^2)\dot{\delta} (t\omega_2-s\omega_1)\wedge (t\tilde{dt}+s\tilde{ds})\\
&+5(c+r^2)^{-1/5} \bigg( 2\cos^2\alpha \left( (t\dot{s}-s\dot{t})\omega_2\wedge\omega_1 +\dot{\delta}(t\omega_2-s\omega_1)\wedge (t\tilde{dt}-s\tilde{ds})\right)\\
&+2\sin\alpha\cos^2\alpha \left( ts\dot{\delta}(t\omega_2-s\omega_1)\wedge \sigma_2+(s\dot{t}-t\dot{s})(t\omega_2-s\omega_1)\wedge \sigma_3\right)\\
&+2\sin\alpha\cos\alpha(s^2-t^2)\dot{\delta}(t\omega_2-s\omega_1)\wedge d\alpha+(t^2+s^2)\sin^3\alpha (t\omega_2-s\omega_1)\wedge (\dot{\alpha}d\beta-\dot{\beta}d\alpha)\\
&+4\cos\alpha\sin^2\alpha \left(\dot{\beta}(t\omega_2-s\omega_1)\wedge (s\tilde{ds}+t\tilde{dt})-(s\dot{s}+t\dot{t})(t\omega_2-s\omega_1)\wedge d\beta\right)\bigg),
\end{aligned}\\
\Psi_2&=\begin{aligned}[t]
&25(c+r^2)^{4/5} \big(-2\cos^3\alpha\sin\alpha st\dot{\delta}\sigma_2\wedge d\beta+\cos^3 \alpha\sin\alpha \dot{\beta}\sigma_2\wedge (t\omega_1+s\omega_2)\\
&+2\cos^3 \alpha (s\dot{t}-t\dot{s})\sigma_2\wedge d\alpha +2\cos^3 \alpha \dot{\alpha} \sigma_2\wedge (t\tilde{ds}-s\tilde{dt})-\cos^3 \alpha\sin^2\alpha(t^2+s^2)\dot{\beta}\sigma_2\wedge \sigma_3\big)\\
&+10(c+r^2)^{\!-1/5}\! \bigg(2\sin\alpha	\cos^2\alpha(t^2+s^2)\sigma_2\!\wedge\! (\dot{t}\tilde{ds}-\dot{s}\tilde{dt})\!-\!2\cos^2\alpha\sin\alpha (s\dot{s}+t\dot{t}) (t\tilde{ds}-s\tilde{dt})\!\wedge\! \sigma_2\\
&-2\cos\alpha\sin^2\alpha st\dot{\delta}(t\tilde{ds}-s\tilde{dt})\wedge d\beta+\cos\alpha\sin^2\alpha \dot{\beta}(t\tilde{ds}-s\tilde{dt})\wedge (t\omega_1+s\omega_2)\\
&+2\cos\alpha \sin\alpha (s\dot{t}-t\dot{s})(t\tilde{ds}-s\tilde{dt})\wedge d\alpha-\cos\alpha\sin^2\alpha (t^s+s^2)\dot{\alpha}\sigma_2\wedge (t\tilde{ds}-s\tilde{dt})\\
&+\cos\alpha\sin^3\alpha (t^2+s^2)\dot{\beta}\sigma_3\wedge (t\tilde{ds}-s\tilde{dt})\bigg) +8(c+r^2)^{-6/5} \sin^2\alpha(t^2+s^2)(s\dot{t}-t\dot{s})\tilde{ds}\wedge \tilde{dt},
\end{aligned}\\
\Psi_3&=\begin{aligned}[t]
&25(c+r^2)^{4/5} \big(-2\cos^3\alpha st\dot{\delta} \sigma_3\wedge d\alpha +\cos^3\alpha \dot{\alpha} \sigma_3\wedge (s\omega_2+t\omega_1)+2\cos^3\alpha\sin\alpha(t\dot{s}-s\dot{t})\sigma_3\wedge d\beta \\
&+2\cos^3\alpha \sin\alpha \dot{\beta}\sigma_3 \wedge (s\tilde{dt}-t\tilde{ds})+(s^2+t^2)\cos^3\alpha\sin^2\alpha \dot{\beta} \sigma_3\wedge \sigma_2\big)\\
&-4(c+r^2)^{-6/5} \sin^2\alpha\dot{\delta} st (t\omega_1+s\omega_2)\wedge (t\tilde{dt}+s\tilde{ds})\\
&+5(c+r^2)^{-1/5}\bigg( 2\sin\alpha\cos^2\alpha\left((t\dot{t}+s\dot{s})\sigma_3\wedge (t\omega_1+s\omega_2)-2\dot{\delta}st\sigma_3\wedge(t\tilde{dt}+s\tilde{ds})\right)\\
&-2\sin\alpha\cos^2\alpha (s\dot{s}+t\dot{t})(t\omega_1+s\omega_2)\wedge \sigma_3-2\cos\alpha\sin\alpha st\dot{\delta}(t\omega_1+s\omega_2)\wedge d\alpha\\
&+2\cos\alpha\sin^2\alpha (t\dot{s}-s\dot{t})(t\omega_1+s\omega_2)\wedge d\beta +2\cos\alpha\sin^2\alpha \dot{\beta} (t\omega_1+s\omega_2)\wedge (s\tilde{dt}-t\tilde{ds})\\
&+(s^2+t^2)\cos\alpha\sin^2\alpha \dot{\alpha}(t\omega_1+s\omega_2)\wedge \sigma_3+(s^2+t^2)\cos\alpha \sin^3\alpha \dot{\beta}(t\omega_1+s\omega_2)\wedge \sigma_2 \bigg),
\end{aligned}\\
\Psi_4&=\begin{aligned}[t]
&2(c+r^2)^{-6/5} \sin^2\alpha (t^2+s^2)\left(\dot{\delta} (s\omega_1+t\omega_2)\wedge (t\tilde{dt}+s\tilde{ds})+4(\dot{t}s-\dot{s}t)\tilde{dt}\wedge\tilde{ds}\right)\\
&+50 (c+r^2)^{4/5}\left( -\cos^2\alpha (\dot{\alpha}d\alpha+\sin^2\alpha \dot{\beta}d\beta)\wedge (s\tilde{ds}+t\tilde{dt})+\cos\alpha\sin^3\alpha (t^2+s^2) \dot{\beta} d\beta\wedge d\alpha \right)\\
&+10(c+r^2)^{-1/5}\bigg( \sin^2\alpha (t^2+s^2)(\dot{\alpha}d\alpha+\sin^2 \alpha \dot{\beta}d\beta)\wedge (t\tilde{dt}+s\tilde{ds})-4\cos^2\alpha(\dot{s}t-\dot{t}s)\tilde{ds}\wedge \tilde{dt}\\
&+4\sin\alpha\cos\alpha(t^2+s^2)(\dot{s}\tilde{ds}+\dot{t}\tilde{dt})\wedge d\alpha-\cos^2\alpha \dot{\delta}(s\omega_1+t\omega_2)\wedge (s\tilde{ds}+t\tilde{dt})\\
&+\sin\alpha\cos\alpha(t^2+s^2)\dot{\delta}(s\omega_1+t\omega_2)\wedge d\alpha\bigg),
\end{aligned}
\end{align*}
where $r^2=s^2+t^2$.

Moreover, 
$$\eta=\pi_7(\Psi_1+\Psi_2+\Psi_3+\Psi_4),$$
where $\eta$ and $\pi_7$ are defined in Proposition \ref{Cayley condition KM}.
\end{corollary}

Finally, we turn our attention to the map $\pi_7$. As recalled in Remark \ref{Remark explanation pi_7}, this map is the projection to the linear subspace $\Lambda^2_7$ of the space of $2$-forms on $M$.

\begin{lemma}\label{Basis Lambda^2_7 SO(3)xId_2}
In the coframe $\{\sigma_2, \sigma_3, d\alpha,d\beta, \omega_1,\omega_2,\tilde{ds}, \tilde{dt}\}$, a basis for $\Lambda^2_7$ is given by the following $2$-forms:
\begin{align*}
\lambda_1&:=-\cos\alpha\sigma_2\wedge\omega_1+d\alpha\wedge\omega_2+2\sin\alpha d\beta\wedge \tilde{dt}+2\cos\alpha \sigma_3\wedge \tilde{ds},\\
\lambda_2&:=\cos\alpha\sigma_2 \wedge \omega_2+d\alpha\wedge \omega_1-2\sin\alpha d\beta\wedge \tilde{ds}+2\cos\alpha \sigma_3\wedge \tilde{dt},\\
\lambda_3&:=\cos\alpha\sigma_3\wedge \omega_1 +\sin\alpha d\beta\wedge \omega_2+2\cos\alpha \sigma_2\wedge \tilde{ds}-2d\alpha\wedge \tilde{dt},\\
\lambda_4&:=-\cos\alpha\sigma_3\wedge\omega_2+\sin\alpha d\beta\wedge \omega_1+2\cos\alpha \sigma_2 \wedge \tilde{dt} +2d\alpha\wedge \tilde{ds},\\
\lambda_5&:=5(c+r^2)\cos\alpha \sigma_3 \wedge d\alpha +5(c+r^2)\sin\alpha\cos\alpha \sigma_2\wedge d\beta+2\omega_2 \wedge \tilde{ds}+2\omega_1\wedge \tilde{dt},\\
\lambda_6&:=5(c+r^2)\sin\alpha\cos\alpha \sigma_3\wedge d\beta -5(c+r^2)\cos\alpha \sigma_2\wedge d\alpha+\omega_2 \wedge \omega_1+4\tilde{dt}\wedge \tilde{ds},\\
\lambda_7&:=5(c+r^2)\sin\alpha d\beta\wedge d\alpha +5(c+r^2)\cos^2\alpha \sigma_3\wedge \sigma_2+2\tilde{ds}\wedge \omega_1-2\tilde{dt}\wedge \omega_2.
\end{align*}
\end{lemma}
\begin{proof}
Using the explicit formula for $\pi_7$ given in Proposition \ref{Cayley condition KM}, it is easy to verify that $\pi_7(\lambda_i)=\lambda_i$ for all $i=1...7$. We deduce that the $\lambda_i$s form a basis of $\Lambda_2^7$ as they are linearly independent and the dimension of $\Lambda_2^7$ is $7$.
\end{proof}

At this point, the proof of Theorem \ref{Complicated ODEs SO(3)xId_2} follows easily. Indeed, we can rewrite the sum of the $\Psi_i$ given in Corollary \ref{definition Psi_i in SO(3)xId_2} as follows:
\begin{align*}
\Psi_1+\Psi_2+\Psi_3+\Psi_4=&5(c+r^2)^{-1/5} \left(-5(c+r^2)\sin\alpha\cos^2\alpha \dot{\beta}t+r^2 \sin^3\alpha \dot{\beta}t-2\sin\alpha\cos\alpha ts^2 \dot{\delta}\right)\lambda_1 \\
&+5(c+r^2)^{-1/5}\!\! \left(5(c+r^2)\sin\alpha\cos^2\alpha \dot{\beta}s\!-\!r^2 \sin^3\alpha \dot{\beta}s-2\sin\alpha\cos\alpha t^2s \dot{\delta}\right)\lambda_2\\
&+5(c+r^2)^{-1/5}\bigg(5(c+r^2)\cos^2\alpha t\dot{\alpha}+4\cos\alpha\sin\alpha t^2\dot{t}+2\sin\alpha\cos\alpha st\dot{s}\\
&+2\sin\alpha\cos\alpha s^2\dot{t}-r^2\sin^2\alpha\dot{\alpha}t\bigg)\lambda_3+5(c+r^2)^{-1/5}\bigg(-5(c+r^2)\cos^2\alpha s\dot{\alpha}\\
&-4\cos\alpha\sin\alpha s^2\dot{s}-2\sin\alpha\cos\alpha st\dot{t}-2\sin\alpha\cos\alpha t^2\dot{s}+r^2\sin^2\alpha\dot{\alpha}s\bigg)\lambda_4\\
&-2\cos^2\alpha st\dot{\delta}\left(25(c+r^2)^{-1/5} \lambda_5\right)\\
&+2\cos^2\alpha(t\dot{s}-s\dot{t})\left(25(c+r^2)^{-1/5} \lambda_6\right)\\
&+2(s^2+t^2)\sin^2\alpha\cos\alpha\dot{\beta}\left(25(c+r^2)^{-1/5} \lambda_7\right).
\end{align*}
From Corollary \ref{definition Psi_i in SO(3)xId_2} and Lemma \ref{Basis Lambda^2_7 SO(3)xId_2}, we deduce the ODEs of Theorem \ref{Complicated ODEs SO(3)xId_2}.

\section{}\label{Appendix analysis ODE Sp(1)xId_1}

In this appendix, we study in detail the ODE (\ref{final ODE Sp(1)xId_1}). First, observe that in the chart we are considering the orbits are $3$-dimensional, hence, the derivative $(\dot\alpha,\dot r)$ can not vanish. In particular, we can reparametrize the curve such that $\dot\alpha=f_1$ and deduce from $(\ref{final ODE Sp(1)xId_1})$ that $\dot r=f_2$. Indeed, we recall that (\ref{final ODE Sp(1)xId_1}) can be rewritten as:
\[
f_1 \dot r-f_2\dot\alpha=0.
\]
Since $(\dot\alpha,\dot r)\parallel(f_1,f_2)$, we recasted the problem into finding the integral curves of the vector field $X=(f_1,f_2)$. Observe that $X$ makes sense on the whole strip $\{r\geq0, \alpha\in[-\pi/2,\pi/2]\} $ and vanishes at $(-\pi/2,0)$ or along the curve $\alpha=\pi/2$. It follows that two solutions of the ODE can only intersect there. Moreover, $\{r=0\}$ and $\{\alpha=-\pi/2\}$ are solutions.

We split our analysis in $3$ parts, corresponding to the different coupled signs of $f_1$ and $f_2$:
\begin{enumerate}
	\item $\alpha\leq\beta_c(r)$;
	\item $\alpha_c(r)\leq\alpha$;
	\item $\beta_c (r)<\alpha<\alpha_c(r)$.
\end{enumerate}

\subsection{The set \texorpdfstring{$\alpha\leq\beta_c(r)$}{Lg}} Since in this set $f_1>0$ and $f_2<0$, starting from an initial point and going forward in time the solution needs to decrease in $r$ and increase in $\alpha$ in a monotonic way, until it hits $\beta_c$. There, $\dot r=0$, so, the solution intersects the curve horizontally. 

If we instead go backwords in time $\alpha$ decreses, while $r$ increases. Hence, the solution can either meet the vertical line $\alpha=-\pi/2$ at some $r_0>0$ or explode at infinity. However, the first instance can not occur since the vertical line $\alpha=-\pi/2$ is a solution of the system of ODEs as well.

\subsection{The set \texorpdfstring{$\alpha_c(r)\leq\alpha$}{Lg}} In this case, we have $f_1<0$ and $f_2>0$,  hence, if we take a point and study the solution going backwards in time the solution needs to decrease in $r$ and increase in $\alpha$. We deduce that it passes through the vertical line $\alpha=\pi/2$ horizontally at some $r_0>0$. Indeed, it does not meet $\{r=0\}$, as the zero section is another solution of the system of ODEs. Moreover, if we reparametrize (\ref{final ODE Sp(1)xId_1}) such that $\dot\alpha=1$, which we can do in the complement of $\alpha_c$, we see that the solution is $r=r(\alpha)$ in this region and that $dr/d\alpha=f_2/f_1<0$. Since $f_2/f_1 \to0$ as $\alpha\to\pi/2$, each solution tends to the vertical line horizontally, and hence, they can not intersect there.

If we go forward in time, we either have $r\to\infty$ or we pass through $\alpha_c$ vertically. Under the same reparametrization as before, we deduce that the solutions $r(\alpha)$ with initial conditions along the line $\{\alpha=\arcsin(-1/4)\}$ can not explode and they need to intersect $\alpha_c$. Moreover, each point of $\alpha_c$ can be reached by such a solution.

\subsection{The set \texorpdfstring{$\beta_c (r)<\alpha<\alpha_c(r)$}{Lg}} As before, we pick a point and we see what happens to the solution going forwards and backwards in time. From the fact that $f_1,f_2>0$, there are only two possibilities forward in time: we either have $r\to+\infty$ as $\alpha\to\arcsin(-1/4)$ or we meet $\alpha_c$ vertically. The latter case will not happen, otherwise, we would have a solution with a cuspid singularity. 

If we go backwards in time, we either intersect $\alpha_c$, $\beta_c$ or $(-\pi/2,0)$. It is obvious that there are solutions intersecting $\alpha_c$ and $\beta_c$. In order to prove the existence of the last case, consider the segment given by an horizontal line restricted to this set. Let $K_{\alpha_c}$ be the subset from which the solutions will meet $\alpha_c$ backwards in time, and let $K_{\beta_c}$ the one relative to $\beta_c$. It is easy to show that these subsets are disjoint connected open subintervals arbitrarily close to each other, using continuity of the initial data. As this can not cover the starting interval we conclude.

            			\bibliographystyle{plain}

\end{document}